\documentclass[11pt]{amsart}

\usepackage{amsmath}
\usepackage{amsthm}
\usepackage{amsfonts}
\usepackage{amssymb}
\usepackage{mathtools}
\usepackage{upgreek}
\usepackage{bm}
\usepackage[dvipsnames,table]{xcolor}
\usepackage[all,cmtip]{xy}
\usepackage{enumerate}
\usepackage{enumitem}
\usepackage{tikz}  
\usepackage{tikz-cd}
\usetikzlibrary{shapes}
\tikzstyle{dot}=[circle,draw,fill=black,inner sep=0pt, minimum width=3pt]
\tikzstyle{wdot}=[circle,draw,fill=white,inner sep=0pt, minimum width=3pt]
\tikzstyle{gdot}=[circle,draw,color=ForestGreen,fill=ForestGreen,inner sep=0pt, minimum width=3pt]
\tikzstyle{pdot}=[circle,color=WildStrawberry,fill=WildStrawberry,draw,inner sep=0pt, minimum width=3pt]
\tikzstyle{bdot}=[circle,draw,color=Plum,fill=Plum,inner sep=0pt, minimum width=3pt, minimum height = 3pt]

\usepackage{graphicx}
\usepackage{scalerel}
\usepackage{microtype}

\usepackage{verbatim}

\usepackage{caption}

\usepackage{libertine}

\usepackage[a4paper,top=3cm,bottom=2.5cm,left=3cm,right=3cm,marginparwidth=1.75cm]{geometry}

\usepackage{pgfplots}
\pgfplotsset{compat=1.15}
\usepackage{mathrsfs}
\usetikzlibrary{shapes.geometric, calc, arrows}
\usepackage[font=small,labelfont=bf]{caption}

\definecolor{xdxdff}{rgb}{0.49019607843137253,0.49019607843137253,1.}
\definecolor{uuuuuu}{rgb}{0.26666666666666666,0.26666666666666666,0.26666666666666666}
\definecolor{ududff}{rgb}{0.30196078431372547,0.30196078431372547,1.}
\definecolor{cite}{RGB}{44,123,182}
\definecolor{ref}{RGB}{215,25,28}

\usepackage[pdftex,
              pdfauthor={Giulia Gugiatti, Franco Rota},
              pdftitle={On the mirrors of low degree del Pezzo surfaces},
              pdfsubject={mirror symmetry, del Pezzo surfaces, pseudolattices},
              pdfkeywords={mirror symmetry, del Pezzo surfaces, pseudolattices},
              colorlinks=true,
              citecolor=cite,
              linkcolor=ref
              ]{hyperref}

\theoremstyle{plain}
\newtheorem{thm}{Theorem}[section]
\newtheorem{corollary}[thm]{Corollary}
\newtheorem{lemma}[thm]{Lemma}

\newtheorem{proposition}[thm]{Proposition}
\newtheorem*{thm*}{Theorem}
\newtheorem*{corollary*}{Corollary}
\newtheorem*{lemma*}{Lemma}
\newtheorem*{ld*}{Lemma/Definition}
\newtheorem*{proposition*}{Proposition}
\theoremstyle{definition}
\newtheorem{definition}[thm]{Definition}
\newtheorem{remark}[thm]{Remark}
\newtheorem{example}[thm]{Example}

\newtheorem*{definition*}{Definition}
\newtheorem*{remark*}{Remark}
\newtheorem*{example*}{Example}
\newtheorem*{xca*}{Exercise}
\newtheorem*{claim*}{Claim}
\newtheorem*{fact*}{Fact}
\newtheorem*{notation*}{Notation}
\newtheorem*{construction*}{Construction}
\newtheorem*{ack*}{Acknowledgements}
\newtheorem*{question*}{Question}
\newtheorem*{problem*}{Problem}
\newtheorem*{conjecture*}{Conjecture}
\newtheorem*{assumption*}{Assumption}

\newcommand{\C}{\mathbb{C}}
\newcommand{\CC}{\mathbb{C}}
\newcommand{\Z}{\mathbb{Z}}
\newcommand{\ZZ}{\mathbb{Z}}
\newcommand{\Q}{\mathbb{Q}}
\newcommand{\R}{\mathbb{R}}

\newcommand{\pr}[1]{\mathbb P^{#1}}

\usepackage[mathscr]{euscript}

\docsvlist{A,B,C,D,E,F,G,H,I,J,K,L,M,N,O,P,Q,R,S,T,U,V,W,X,Y,Z} %Macro for caligraphic e.g. \cE=\mathcal{E}

\docsvlist{A,B,C,D,E,F,G,H,I,J,K,L,M,N,O,P,Q,R,S,T,U,V,W,X,Y,Z} %Macro for script e.g. \sE=\mathscr{E}

\docsvlist{A,B,C,D,E,F,G,H,I,J,K,L,M,N,O,P,Q,R,S,T,U,V,W,X,Y,Z} %Macro for roman e.g. \rE=\mathrm{E}

\docsvlist{A,B,C,D,E,F,G,H,I,J,K,L,M,N,O,P,Q,R,S,T,U,V,W,X,Y,Z,
a,b,c,d,e,f,g,h,i,j,k,l,m,n,o,p,q,r,s,t,u,v,w,x,y,z} %Macro for sans serif e.g. \rsf=\mathsf{E}

\docsvlist{A,B,C,D,E,F,G,H,I,J,K,L,M,N,O,P,Q,R,S,T,U,V,W,X,Y,Z,
a,b,c,d,e,f,g,h,i,j,k,l,m,n,o,p,q,r,s,t,u,v,w,x,y,z} %Macro for math boldface e.g. \bfp=\mathbf{p}

\DeclarePairedDelimiter{\pair}{\langle}{\rangle}

\DeclareMathOperator{\id}{id}
\DeclareMathOperator{\rk}{rk}

\DeclareMathOperator{\Ker}{Ker}

\DeclareMathOperator{\Hom}{Hom}

\DeclareMathOperator{\Aut}{Aut}

\DeclareMathOperator{\td}{td}

\DeclareMathOperator{\im}{Im}

\DeclareMathOperator{\NS}{NS\,}

\newcommand{\lagvc}{\mathrm{Lag}_{\mathrm{vc}}}

\DeclareMathOperator{\Coh}{Coh}

\DeclareMathOperator{\pairing}{\left\langle \cdot\,,\cdot \right\rangle}
\newcommand{\Yt}{\widetilde{Y}}
\newcommand{\wt}{\widetilde{w}}

\newcommand{\FS}{\mathrm{FS}}

\begin{document}
\title[]{On the 
mirrors of \\ 
low degree del Pezzo surfaces}

\author[G. Gugiatti]{Giulia Gugiatti}
\address{GG: School of Mathematics\\ University of Edinburgh\\ Edinburgh EH9 3FD, United Kingdom} 
\email{giulia.gugiatti@ed.ac.uk}

\author[F. Rota]{Franco Rota}
\address{FR: Laboratoire de Math\'emathiques d'Orsay, Universit\'e Paris-Saclay, Rue Michel Magat, B\^at. 307, 91405 Orsay, France} 
\email{franco.rota@universite-paris-saclay.fr}

\thanks{
GG acknowledges support from the ERC Grant 819864, the ERC Starting Grant 850713, and a Simons Investigation Award (n. 929034).
FR is supported by the European Union’s Horizon 2020 Research and Innovation Programme under the  Marie Sk\l{}odowska-Curie grant agreement n. 101147384 (\href{https://cordis.europa.eu/project/id/101147384}{CHaNGe}). He acknowledges support from the EPSRC grant EP/R034826/1 and from the Deutsche Forschungsgemeinschaft (DFG, German Research Foundation) under Germany's Excellence Strategy -- EXC-2047/1 -- 390685813.
}


%%%old
%Primary 14F08, %Derived categories of sheaves, dg categories, and related constructions in algebraic geometry
%14A20, %Generalizations (algebraic spaces, stacks)
%14E16, %McKay correspondence, 
%14J45; %Fano varieties
%secondary 14J33 %mirror symmetry

%%%new
%14F08  	Derived categories of sheaves, dg categories, and related constructions in algebraic geometry

%14J17  	Singularities of surfaces or higher-dimensional varieties
%14J26  	Rational and ruled surfaces
%14J27  	Elliptic surfaces, elliptic or Calabi-Yau fibrations
%14J33  	Mirror symmetry (algebro-geometric aspects)
%53D37  	Symplectic aspects of mirror symmetry, homological mirror symmetry, and Fukaya category

%14J45   Fano varieties

\begin{abstract}
We compare different constructions of mirrors of del Pezzo surfaces, focusing on degree $d \leq 3$. %low degree cases. 
In particular, we extract Lefschetz fibrations, with associated exceptional collections, from the mirrors obtained via the Hori--Vafa and Fanosearch program constructions, which we
%survey and 
relate to one another. 
We show with geometric methods that the Lefschetz fibrations define categorical mirrors.
With a more explicit approach, we give a sequence of (numerical) mutations relating the exceptional collections considered by Auroux, Katzarkov, and Orlov
with those arising in this paper. 
This uses the theory of surface-like pseudolattices, and extends some of the string junction results of Grassi, Halverson and Shaneson. Our argument lifts directly to an equivalence of certain Fukaya--Seidel categories %of exact symplectic Lefschetz fibrations. 
arising from our fibrations and those of Auroux, Katzarkov, and Orlov.
\end{abstract}

\maketitle

\setcounter{tocdepth}{1}
\tableofcontents

\section{Introduction}
In this paper we compare different mirror constructions for del Pezzo surfaces of degree at most three.
Mirror symmetry predicts that a Fano manifold $X$ admits as a mirror a LG model $(Y,w)$, that is, a pair formed by a non-compact $\cC^\infty$ manifold $Y$ together with a $\cC^\infty$ function $w\colon Y \to \C$ called the superpotential. 
%%%dim Y not specified
Generally speaking, both $X$ and $(Y,w)$ carry a symplectic structure and a complex structure -- often called, respectively, their A-side and their B-side. 

To say that $X$ is mirror to $(Y,w)$ means that:
\begin{itemize}
    \item[(i)] the A-side of $X$ corresponds to the B-side of $(Y,w)$, and
    \item[(ii)] the B-side of $X$ corresponds to the A-side of $(Y,w)$.
\end{itemize}

Statements (i) and (ii) can be given a precise meaning at different levels.
A formulation of (i), which we call \textit{cohomological}, 
identifies two cohomological invariants: the regularised quantum period of $X$, and a period of $(Y,w)$. In this case, we say that $X$ and $(Y,w)$ are cohomological mirrors. 
The Homological Mirror Symmetry (HMS) \cite{K1994} formulation of (ii), which we call \emph{categorical}, consists of an equivalence $D^b(\Coh(X))\simeq \mathrm{FS}(Y,w)$ between the bounded derived category of coherent sheaves on $X$ and the Fukaya--Seidel category of $(Y,w)$ \cite{Kon_Lectures, HV-Lag, Seidel} (we refer to $X$ and $(Y,w)$ as to categorical mirrors).  

Cohomological mirrors to del Pezzo surfaces can be obtained from the Fanosearch Program, 
or
equivalently from the Givental/Hori--Vafa construction, 
see Section \ref{sec:general}. They can be represented by Laurent polynomials in two variables.
Categorical mirrors to del Pezzo surfaces are constructed in \cite{Seidel-more, AKO06}. 
They are genus-one Lefschetz fibrations
that extend
to rational elliptic surfaces with a fibre of type $I_d$ over $\infty$, where $d$ is the degree of the del Pezzo surface.

This paper shows that, for del Pezzo surfaces $X_d$ of degree $d \leq 3$, the cohomological mirrors described above also give rise to categorical mirrors. We keep our approach as constructive and explicit as possible. 
The constructions we run below apply, in fact, 
to del Pezzo surfaces of any degree $d$. 
We comment at the end of the Section on our choice of focusing on $d \leq 3$: on the one hand, this is a starting point for future research objectives; on the other hand, our result gives a (non-commutative) answer to representation theoretic questions posed in string theory.

\subsection*{Main results} To compare the mirrors 
to $X_d$, it is useful to consider superpotentials with proper fibres.
The cohomological mirrors to $X_d$ described above give rise to 
rational elliptic surfaces $F_d \colon S_d \to \pr 1$ with singular fibres types:
\begin{equation}
    \label{eq:fibres}
\begin{array}{|c||c|c|c|}
%k & 0 & \lambda_{0} & \infty\\
 \hline
d=1 &  \mathrm{II}^\star & I_1 & I_1 \\
\hline
d=2 & \mathrm{III}^\star  & I_1 & I_2\\
\hline
d=3 & \mathrm{IV}^\star & I_1 & I_3\\
\hline
\end{array}
\end{equation}
The corresponding Weiestrass fibrations have a singularity of type $E_8$, $E_7$, $E_6$ over the first critical value.

The existence of these elliptic surfaces is well established in the literature.  
A construction of $F_d \colon  S_d \to \pr 1$, due to 
Tveiten \cite{Tve18},  starts from a representative Laurent polynomial mirror. We sketch in Section \ref{sec:elliptic-surf}. 
In the same section, taking a different approach, we start from the Givental/Hori–Vafa construction and we extract a local Weierstrass form for 
$F_d\colon S_d\to \pr 1$.

\begin{lemma}[Lemma \ref{lem:Wf1}]
\label{lem:Wf}    Equations \eqref{eq:Weierstrass-eqn} are local Weierstrass forms for the elliptic surfaces $F_d\colon S_d \to \pr 1$. 
\end{lemma}

Restricting 
$F_d$ to the preimage of $\CC \subset \mathbb{P}^1$ we obtain an LG model, which we denote by $(Y_d, w_d)$. This LG model is a cohomological mirror to $X_d$ by construction.
It is natural to ask whether $(Y_d,w_d)$ also defines  a categorical mirror. 

One approach is to argue that, once the complex structure is disregarded, $(Y_d,w_d)$ is equivalent to the LG model built in \cite{AKO06}. We comment on this in Remark \ref{rem:Y-mirror}.
Otherwise, one may study the derived category $\mathrm{FS}(Y_d,w_d)$ directly. 
When the superpotential of an LG model $(Y,w)$ has isolated critical points, the category $\mathrm{FS}(Y,w)$ coincides with the better-understood derived category $D^b\lagvc(\hat{w})$ of Lagrangian vanishing cycles of a Morsification $\hat{w}$ of $w$ \cite{Seidel}. 
While the map $w_d$ has non-isolated singularities, we still take a similar approach and consider a perturbation $(\Yt_d,\wt_d)$ of $(Y_d,w_d)$ which defines a Lefschetz fibration.

To construct $(\Yt_d,\wt_d)$, we proceed as follows.
As a first step, we build an explicit rational elliptic surface whose fibres over $\CC$ are all of type $I_1$ and whose fibre type over $\infty$ matches that of $F_d$. We obtain it by a suitable perturbation of  Equation \eqref{eq:Weierstrass-eqn}, given in Equation \eqref{eq:dP-def}. 
\begin{lemma}[Lemma \ref{lem:lemma}]
    \label{lem:short-lemma}
    Equation \eqref{eq:dP-def} defines a rational elliptic surface with $12-d$ fibres of type $I_1$ over $\CC$, and a fibre of type $I_d$ over $\infty$. 
\end{lemma}
Restriction to $\CC$ produces the LG model $(\widetilde{Y}_d, \widetilde{w}_d)$.
In Section \ref{sec:geometry-cats} we show by geometric methods that $(\Yt,\wt)$ is equivalent to the LG model built in \cite{AKO06}, hence it is a categorical mirror to $X_d$ (see Proposition \ref{prop:EquivalenceYtildeAko}).

Arguing more concretely, in Section \ref{sec:PseudolatticesAnd Mutations} we exhibit an isometry between the Grothendieck groups of $D^b(\Coh(X_d))$ and $\mathrm{FS}(\wt_d)$ obtained as a sequence of numerical mutations (Definition \ref{def:num-mutation}). We proceed as follows.

Since the map $\widetilde{w}_d$ is a Lefschetz fibration,  
the category $\mathrm{FS}(\widetilde{Y}_d, \widetilde{w}_d)$ coincides with  $D^b\lagvc(\widetilde{w}_d)$.
Both $D^b\Coh(X_d)$ and $D^b\lagvc(\widetilde{w}_d)$ admit full exceptional collections,  hence their Grothendieck groups, which we denote by $K_0(X_d)$ and $K_0(\widetilde{w}_d)$, coincide with the numerical Grothendieck groups. 
In Section \ref{sec:construct-vc},  using bifibration techniques
\cite[III.(15e))]{Seidel}
,  we construct exceptional bases
%full exceptional collections 
for 
%the Grothendieck groups 
$K_0(\widetilde{w}_d)$
%$D^b\lagvc(\widetilde{w}_d)$
and describe 
the homology classes of the corresponding vanishing cycles in $H_1(\widetilde{w}_d^{-1}(0), \ZZ)$. 
\begin{proposition}[Proposition \ref{pro:vc}]\label{pro:coho-classes} The sequences \eqref{eq:vc-dP1}, \eqref{eq:vc-dP2}, and \eqref{eq:vc-dP3} yield the homology classes of the vanishing cycles corresponding to the chosen exceptional bases.
\end{proposition}

Using pseudolattice techniques, recalled in Section \ref{sec:surfacelikePseudolatt-prelim}, 
we produce %an explicit isometry $K_0(X_d) \simeq K_0(\widetilde{w}_d)$ given by 
the desired sequence of numerical mutations.

\begin{thm}[Theorem \ref{thm:MutationSequence}]
    \label{thm:short-thm}
The basis for $K_0(X_d)$ given in \cite[Equation (2.3)]{AKO06} and the basis constructed in Section \ref{sec:construct-vc} for $K_0(\widetilde{w}_d)$ are related by the
sequences of numerical mutations \eqref{eq:mutationsFord1}, \eqref{eq:mutationsFord2}, \eqref{eq:mutationsFord3} (for $d=1,2,3$). \end{thm}
In Section \ref{sec:thm} we discuss how Theorem \ref{thm:short-thm} provides a tool to obtain a
%an explicit 
numerical mutation sequence for any given LG model diffeomorphic to $(\widetilde{Y}_d, \widetilde{w}_d)$ whose superpotential is a Lefschetz fibration.

In Section \ref{sec:lifts}, we lift the sequence of mutations in Theorem \ref{thm:short-thm} to an equivalence between the derived categories of Lagrangian vanishing cycles of certain exact symplectic Lefschetz fibrations.
These are denoted $(G_d,q_d,\varpi_d)$ and $(G'_d,q'_d,\varpi'_d)$, where $q_d \colon G_d \to \C$ is obtained by removing the distinguished section of $\wt_d$ and $q'_d\colon G'_d \to \C$ is obtained in the same way from the mirror in \cite{AKO06}. The symplectic forms $\varpi_d, \varpi'_d$ are the restriction of forms whose class is Poincar\'e dual to the section. 

\begin{thm}[Theorem \ref{thm:equivOfCategories} and Corollary \ref{cor:equivOfCategories}]
With the choice of symplectic forms above, there  is an equivalence 
\[ D^b\lagvc(q_d) \xrightarrow{\sim} D^b\lagvc(q'_d)\]
which maps the exceptional collections in Section \ref{sec:construct-vc} to those given in \cite[Equation (2.3)]{AKO06}.
\end{thm}

We comment on the relation between $D^b\lagvc(q_d)$ and $D^b\lagvc(\wt_d)$ in Remark \ref{rem:BulkDef}.

\subsection*{Notes on the construction}
%Grandi seghe
%Unlike \cite{Loo81, Friedman15}, our construction yields an explicit deformation of the elliptic surface $F_d \colon S_d \to \pr 1$ into the elliptic surface of Lemma \ref{lem:short-lemma}; moreover, the deformation is such that every member of the family remains an elliptic surface. We expand this discussion in Remark \ref{rem:def-Wf-es}. 
The choice of the perturbation \eqref{eq:dP-def} is modeled on \cite[Equations (3.9) and (3.11)]{GHS13:matter} and allows us to recover the collections of vanishing cycles appearing in \cite[Section 3.4]{GHS13:matter}.
While
the exceptional collections \cite[Equation (2.3)]{AKO06} reflect the structure of $X_d$ as blow-ups of $\pr 2$ in $9-d$ points, the ones we extract out of $(\widetilde{Y}_d, \widetilde{w_d})$ 
highlight the relation between $X_d$ and the representation theory of root systems of type $E_{9-d}$, in the spirit of \cite{DWZ,GHS13:matter}.
We expand on this below and in Section \ref{sec:GHS}.

Even without using that $(\Yt_d,\wt_d)$ is a categorical mirror, the existence of an isometry between $K_0(X_d) $ and $K_0(\widetilde{w}_d)$ is guaranteed by \cite[Theorem 5.1]{HarderThompson20}.

\subsection*{Sketch of the proof of Theorem \ref{thm:short-thm}}
We first determine the correct mutation sequence for $d=3$ by making use of the theory of surface-like pseudolattices \cite{Kuz17_pseudolattices, Vial17,VdB_dTdV}.
Rather than following the same approach for $d=2$, we directly derive the mutation sequence for $d=2$ from that of $d=3$: this is achieved by interpolating the local Weierstrass forms \eqref{eq:dP-def} for $d=3$ and $d=2$
and mutating the newly created vanishing cycle into the last position.
We obtain the mutation sequence for $d=1$ from that of $d=2$ by the same approach. This method views the three Weierstrass fibrations of the elliptic surfaces in Lemma \ref{lem:short-lemma} as algebraic deformations of one another, which
 reflects the fact that the three del Pezzo surfaces are blow-ups of one another. It can be viewed as an
algebraic version (for $d=1,2,3$) of the approach used in \cite[Section 3.3]{AKO06} to construct categorical mirrors to del Pezzo surfaces $X_d$ that are blow-ups of $\pr 2$ from the mirror of $\mathbb{P}^2$, and it is analogous to the \textit{radar screens} introduced in \cite{DKK12}.  

\subsection*{Motivation and applications}
Our focus on del Pezzo surfaces of degree $d \leq 3$ is motivated by two key reasons.

Firstly, this work is a step towards the goal of showing HMS for anticanonical log del Pezzo weighted hypersurfaces (classified in \cite[Theorem 8]{JK}). Indeed, del Pezzo surfaces of degree $\leq 3$ 
arise as the only smooth anticanonical del Pezzo weighted hypersurfaces.
Del Pezzo surfaces of degree $d=2$ are also the degenerate case ($k=0$) of the so-called Johnson--Koll\`ar series $X_{8k+4} \subset \mathbb{P}(2,2k+1, 2k+1, 4k+1)$, ($k\geq 1, k \in \mathbb{N}$). Cohomological mirrors $(Y_k,w_k)$ (where $Y_k$ is a surface) to the Johnson-Koll\`ar series  are built in \cite{ACGG}: in the case $k=0$ they coincide with the cohomological mirrors arising from the classical constructions (see Remark \ref{rem:jk}), but the construction of \cite{ACGG} is the only one known for general $k$. It will be the goal of a future study to show that they also define categorical mirrors. See \cite{GR22} for a study of the category of coherent sheaves on (the canonical stack of) $X_{8k+4}$.

Secondly, our focus on $d\leq 3$ is inspired by \cite{GHS13:matter}, which connects ADE lattices
to the Picard--Lefschetz theory 
of certain smooth deformations of ADE singularities.

In the language of this paper, the authors of \cite{GHS13:matter} view the  local equation 
\begin{equation*}
    y^2=x^3+a(z)x+b(z)
\end{equation*}
of an ADE singularity as a Weierstrass form defining a
genus-one fibration $Z \to D$ over a disc, with discriminant $z^N$ (here $N=\ell+1$ for $A_\ell$, while $N=\ell+2$ for $D_\ell$ and $E_\ell$),
and deform \eqref{eq:GHSSing} to a genus-one Lesfchetz fibration $\widetilde{q} \colon \widetilde{Z} \to D$ with $N$ critical values. 
To this fibration, the authors attach a rank $N$ lattice $\bfJ$ of \textit{string junctions}, together with a specific basis of $\bfJ$ consisting of Lefschetz thimbles. They then exhibit a copy of the ADE root lattice in $\bfJ$ which lies in the kernel of a homomorphism called the \textit{asymptotic charge} $c_{\widetilde{q}}$. Without performing an in-depth study of the exceptional $E_\ell$ cases, the authors then ask whether string junctions also encode the full structure of ADE algebras, including their weight lattices \cite[Section 3.5]{GHS13:matter}.

In Section \ref{sec:GHS}, we give convenient descriptions of the junction lattices and, in the E case,  answer (a version of) their question in the positive.
This is done by slightly enlarging $\bfJ$ and using a different pairing, called the \textit{Seifert pairing}, which is crucially non-symmetric. On the one hand, the Seifert pairing recovers the junction pairing of \cite{GHS13:matter} after symmetrization. On the other,  $\bfJ$, with the Seifert pairing, is isometric to $K_0(\widetilde{q})$. In the E case, this allows us in turn to interpret $\bfJ$ as a sublattice of $K_0(\widetilde{w}_d)$, $d=9-\ell$.
This corresponds to globalising the fibration of \cite{GHS13:matter}
with the LG model $(\widetilde{Y}_d,\widetilde{w}_d)$. It affords us the technology of surface-like pseudolattices, and, through Theorem \ref{thm:short-thm}, a direct comparison with $K_0(X_d)$ on the B-side. In particular, we relate $\bfJ$ with a category of matrix factorizations.

\begin{corollary}[{Corollary \ref{cor:MF}}]\label{conIntro:MF}
   For $d=1,2,3$ let $X_d$ be a del Pezzo surface of degree $d$, defined by a polynomial $h^{(d)}$ as a weighted projective hypersurface. Denote by $\widetilde{q}_{9-d}$ the fibration corresponding to $E_{9-d}$. We have an isometry of pseudolattices
   \[ K_0(\widetilde{q}_{9-d}) \simeq K_0(\mathrm{DGr}(h^{(d)})) \]
   where the latter is the category of graded matrix factorizations of $h^{(d)}$ of \cite[Section 3]{Orlov09_DerSing}. 
\end{corollary}
The isometry of Corollary \ref{conIntro:MF} is (essentially) given by the mutation sequences of Theorem \ref{thm:short-thm}.

Moreover, by working with the B-side counterpart of the asymptotic charge $c_{\wt_d}$ defined on $K_0(\widetilde{w}_d)$, we describe its kernel, generalize the construction on extended weights carried out in \cite{DWZ}, and construct a splitting of $K_0(\wt_d)_\Q$ in terms of $c_{\wt_d}$, analogously to \cite{DWZ}.

We denote by $(c_{\wt_d})_\Q$ the linear extension of $c_{\wt_d}$ to rational coefficients, and write $\bfp$ for a \textit{point-like} object of $K_0(\widetilde{w}_d)$.

\begin{proposition}[Corollary \ref{cor:kernel-dec} and Proposition \ref{pro:H-splitting}]
The kernel $\Ker(c_{\wt_d})$ decomposes as an orthogonal sum $\sfE_{\ell} \oplus \langle \bfp \rangle$. 
Moreover, $K_0(\widetilde{w}_d)$ admits a splitting 
\begin{equation*}
    %\label{eq:splitting}
    K_0(\widetilde{w}_d)_\Q \simeq \Ker((c_{\wt_d})_\Q) \oplus \im((c_{\wt_d})_\Q)
\end{equation*}
which, modulo $\bfp$, is orthogonal with respect to the Seifert pairing.    
\end{proposition}

\subsection*{Plan of the paper}
Section \ref{sec:B-elliptic-surfaces} describes cohomological mirrors to del Pezzo surfaces of degree at most three. General notions are recalled in Section \ref{sec:general}, while specific constructions for Fano weighted complete intersections appear in Section \ref{sec:B-nonproper}. Section \ref{sec:elliptic-surf} constructs the elliptic surfaces $F_d \colon S_d \to \pr 1$ and the associated LG models $(Y_d,w_d)$.
In Section \ref{sec:ConstructionLefschetzFibrations}, after recalling basic notions about Lefschetz fibrations (Section \ref{sec:prelim-Lf}), we construct the LG models $(\widetilde{Y}_d,\widetilde{w}_d)$ (Section \ref{sec:construct-Lf}), we show that they define categorical
mirrors (\ref{sec:geometry-cats}), and we exhibit an exceptional basis of vanishing cycles for $K_0(\widetilde{w}_d)$ (Section \ref{sec:construct-vc}).
Section \ref{sec:PseudolatticesAnd Mutations} begins with a review of the theory of surface-like pseudolattices (Section \ref{sec:surfacelikePseudolatt-prelim}). The rest of the Section is dedicated to Theorem \ref{thm:MutationSequence}  and its proof (Sections \ref{sec:thm} and \ref{sec:thmproof}), along with its applications to string junctions (Section \ref{sec:GHS}). 
Section \ref{sec:lifts} lifts the mutation sequence of Theorem \ref{thm:short-thm} to the categorical level.

\begin{ack*}
We thank Tom Ducat, Liana Heuberger, Wendelin Lutz, and Andrea Petracci
for exchanges on % mirrors to 
del Pezzo surfaces and their mirrors, and   
Ilaria Di Dedda and Matthew Habermann for conversations on symplectic aspects of HMS. 
We thank Antonella Grassi for insights on string junctions, and Luca Giovenzana, Johannes Krah, and Alan Thompson 
for discussing pseudolattice results with us. We thank Nick Sheridan for guidance on the results of Section \ref{sec:lifts}.
Finally, we are grateful to Arend Bayer and Michael Wemyss 
for 
exchanges and opinions on earlier drafts of this manuscript.
\end{ack*}

\subsection*{Open access} For the purpose of open access, the authors have applied a Creative Commons Attribution (CC:BY) licence to any Author Accepted Manuscript version arising from this submission.

\section{Cohomological mirrors and elliptic surfaces}
\label{sec:B-elliptic-surfaces}
This Section focuses on cohomological mirrors to del Pezzo surfaces of degree at most three.

We begin by outlining the notion of mirror symmetry used here (Section \ref{sec:general}), then review the standard constructions of cohomological mirrors to del Pezzo surfaces $X_d$ of degree $d \leq 3$ (Section \ref{sec:B-nonproper}).
These mirrors 
naturally extend to rational elliptic surfaces, the construction of which is detailed in Section \ref{sec:elliptic-surf}.

\subsection{Cohomological mirrors of Fano manifolds} \label{sec:general}
Let $X$ be a Fano manifold of dimension $n$. 
The \emph{quantum period} $G_X$ of $X$ is a power series of the form
$G_X(t)=\sum_{d=0}^\infty p_d t^d$, where $p_0=1$, $p_1=0$ and $p_d$, $d\geq 2$, is a certain genus-zero Gromov--Witten invariant of $X$: roughly speaking, $p_d$ counts the number of degree-$d$ rational curves on $X$ passing through a given point.
%%as in RMMLPS
It is a deformation invariant. 
The \emph{regularised quantum period} $\widehat{G}_X(t)=\sum_{d=0}^\infty d! p_d t^d$ satisfies regular polynomial differential operators over $\CC^\times$. 
We refer to \cite[Section 4]{CC} and \cite[Section B]{Quantum-P-3d} for details.

Let $(Y,w)$ be an LG model of dimension $k$ endowed with a complex structure, i.e. $Y$ is a complex $k$--dimensional manifold and $w$ is a holomorphic
function. % on $Y$. 
We say that $(Y,w)$ is a cohomological mirror to $X$
if there exists a period $\pi$ of $(Y,w)$ such that $\widehat{G}_X=\pi$.
Recall that periods of $(Y,w)$
are integrals of the form:
\begin{equation}
\label{eq:period}
\pi(t)= \int_\Gamma \frac{1}{1-tw} \Omega 
\end{equation}
where $\Gamma \in H_k(Y, \ZZ)$ 
and $\Omega \in H^0(Y, \Omega^k_Y)$ 
is such that $\int_\Gamma \Omega=1$.
We refer the reader 
to \cite{Carlson, Zagier} for a broader discussion of periods.

There is also a notion of quantum period for Fano orbifolds \cite[Section 3.3]{OP}, and the Fano--LG correspondence extends to this setting.  

\begin{example}
Consider the LG model $((\CC^\times)^k, f)$,   where $f$ is a Laurent polynomial in $k$ variables.  The period of $((\CC^\times)^k, f)$ corresponding to the choices \[\Gamma=\{|x_i|=1 \ \forall \ i=1, \dots, k \} \quad \Omega=\left( \frac{1}{2 \pi i} \right)^k \frac{dx_1 \cdots dx_k}{x_1 \cdots x_k}\] is  denoted by $\pi_f$ and called  \emph{the classical period of $f$} \cite[Definition 3.1]{CC}. Its power series expansion is  $\pi_f(t)=\sum_{d=0}^\infty c_0(f^d) t^d$, where $c_0(f^d)$ denotes the constant term of $f^d$. 
Given a Fano manifold $X$, one says that the Laurent polynomial $f$ is a (cohomological) mirror to  $X$ if $\widehat{G}_X=\pi_f$ \cite[Definition 4.9]{CC}. 
Being mirror to some $X$ is a very strong condition on $f$.
\end{example}
\smallskip

For certain Fano toric complete intersections $X$ the quantum period can be computed by the Quantum Lefschetz theorem \cite{Coates-Givental,Wang} 
%smooth implies conditions
and cohomological mirrors of the same dimension as $X$ can be built via the Givental/Hori--Vafa recipe \cite{GIV3, Hori-Vafa}.
Such mirrors have nonproper superpotential
and can be written in the form  $((\CC^\times)^n, f)$, with $n=\dim_\CC X$ and $f$ a  Laurent polynomial  mirror to $X$, 
by the Przyjalkowski method
(see \cite[Section 2]{CKPT}).

\begin{remark}
Del Pezzo surfaces can be realised as complete intersections in toric varieties to which the Givental/Hori--Vafa recipe applies. In particular, del Pezzo surfaces of degree $d \leq 3$ arise as weighted hypersurfaces (see Section \ref{sec:B-nonproper}).
\end{remark}

The setting just discussed is extended by the Fanosearch program, a conjectural framework aiming  to classify Fano varieties through their mirrors. It was first laid out in \cite{CC} and the main recent developements are collected in \cite{CKPT}. 
The program
suggests %more generally 
that any Fano manifold $X$ of dimension $n$ should admit a mirror Laurent polynomial $f$ in $n$ variables. 
If $f$ is mirror to $X$, then $X$ should admit a $\Q$-Gorenstein (qG) degeneration to the toric variety whose fan is the spanning fan of the polytope of $f$.  
Given a Laurent polynomial mirror to $X$, new mirror Laurent polynomials can be obtained from it by a process called \emph{mutation} (see \cite[Section 1]{CKPT}). Two Laurent polynomials are called \emph{mutation-equivalent} if they are  connected by a sequence of mutations.
Conjecturally, the converse also holds, that is, two Laurent polynomials $f$ and $g$ are mirror to the same $X$ if and only if they are mutation-equivalent. 
Moreover, mirrors to Fano manifolds are expected to be special Laurent polynomials known as \emph{rigid maximally mutable Laurent polynomials} (RMMLps), see \cite[Section 2]{CKPT} for the definition. 

Mirror RMMLps exist for all Fano manifolds up to dimension three, and every known mirror Laurent polynomial to a Fano manifold, of any dimension, is a RMMLp. 
Moreover, all conjectural statements discussed above have been proven for del Pezzo surfaces,
see \cite{Cor15} and \cite[Section 3]{CKPT} for details.

\subsection{Cohomological mirrors to the surfaces $X_d$}
\label{sec:B-nonproper}
This section is organised as follows. First we expand on the Givental/Hori--Vafa recipe and on the Przyjalkowski method for a weighted complete intersection. Then we apply these constructions
%give HV and Laurent polynomial mirrors 
to del Pezzo surfaces $X_d$ of degree $d=1,2,3$.

\subsubsection{Fano weighted complete intersections}
\label{sec:Fano-wci}
Let $X \subset \mathbb{P}(a_1, \dots, a_m)$ be a %smooth
quasismooth wellformed 
Fano weighted complete intersection of multidegree $(d_1, \dots, d_c)$ with canonical singularities. Let $n=m-1-c$ be the dimension of $X$. Let $\iota=\sum_{i=1}^{m} a_i -\sum_{i=1}^c d_i>0$ be the Fano index. The quantum period of $X$ satisfies:
\begin{equation}
\label{eq:G-I}
    G_X(t)=e^{-\alpha t} \cdot 
    \sum_{j=0}^{\infty}  \frac{\prod_{i=1}^c (d_i \cdot j)!}{\prod_{i=1}^m (a_i \cdot j)!} \; t^{\iota \cdot  j}
    %I_X (t)
\end{equation} 
where $\alpha$ is the unique rational number such that the right hand side has the form $1 +\mathcal{O}(t^2)$. 
%Note that the series $\widehat{I}_X$ .....

\subsubsection*{Givental/Hori--Vafa recipe} Assume now that $X$ is such that: 
\begin{equation}
\label{eq:HV-condition}
    \forall i \in \{1, \dots, c\} \ \exists \ S_i \subset \{1, \dots, m\} \ \text{such that} \  d_i=\sum_{j \in S_i} a_j \ \text{and} \ S_h \cap S_l=\varnothing \ \text{for} \  h \neq l
\end{equation}
Let $(\CC^\times)^m$ be an $m$-dimensional torus with coordinates $x_1, \dots x_m$. 
The Givental/Hori--Vafa   mirror (HV mirror) to $X$ is the LG model $(Y^\circ,v^\circ)$, where $Y^\circ$ is the $n$-dimensional manifold:
\begin{equation}
\label{eq:HV}
Y^\circ=\left(\prod_{i=1}^m x_i^{a_i}=1, \sum_{j \in S_i} x_j=1  \ \forall i=1, \dots, c\right) \subset (\CC^\times)^m
\end{equation} and $v^\circ \colon Y \to \CC$ is the map \begin{equation}
v^\circ(x_1, \dots, x_m)=
%\sum_{i=1}^m x_i -c=
\sum_{j \notin \cup_{i=1}^c S_i} x_j
\end{equation} 

The HV mirror is a cohomological mirror of $X$ after translation of the base by $\alpha$ \cite[Remark 2.9]{LInversion}. 
In other words, the LG model $(Y^\circ,w^\circ)$, where $w^\circ=v^\circ -\alpha$, is a cohomological mirror of $X$.

\subsubsection*{Przyjalkowski method} 
We recall the Przyjalkowski method in this setting. 

Assume that $X$ satisfies \eqref{eq:HV-condition} and let $(Y^\circ, v^\circ)$ be its HV mirror. 
For each $j \notin  \cup_{i=1}^c S_i$, introduce a new variable $y_j$. For 
    each $i \in \{ 1, \dots, c\}$, pick an element $s_i \in S_i$, set $y_{s_i}=1$, and introduce new variables $y_j, j \in S_i \setminus\{s_i\}$. Then
    set
\begin{equation}
\label{eq:P-method-yj}
 \begin{split}
& x_j=\frac{y_j}{\sum_{l \in S_i} y_l} \quad j \in S_i \\
&x_j=y_j  \quad \quad \quad \quad  j \notin \cup_{i=1}^c S_i
\end{split}
\end{equation} 
The first set of identities in \eqref{eq:P-method-yj} solve the constraints $\sum_{j \in S_i} x_j=1$ in \eqref{eq:HV}. Eliminate an additional variable $y_j, j \notin \cup_{i=1}^c S_i$, by the first equation in \eqref{eq:HV}. Then 
$v^\circ$ becomes a Laurent polynomial $g$ in $n=m-1-c$ variables,
and $\varpi_\alpha=\pi_{g-\alpha}$.

\begin{remark}
\label{rem:mutation-eq} 
The Laurent polynomial $g-\alpha$ is a mirror to $X$, thus it is expected to be mutation-equivalent to any Laurent polynomial mirror to $X$. 
\end{remark}

\begin{remark}
\label{rem:mutation-eq-S} 
The choice of sets $S_i$ in \eqref{eq:HV-condition} is not necessarily unique. If $g_1$ and $g_2$ are the Laurent polynomials obtained by the Przyjalkowski method from the two different choices of sets $S_i$, then (as a special case of the previous Remark) $g_1 -\alpha$ and $g_2-\alpha$ are expected to be mutation-equivalent.
\end{remark}

%\AAA{Nuova versione, con meno $k$.}
\subsubsection{The surfaces $X_d$}
\label{sec:Xd-HV}
Del Pezzo surfaces $X=X_d$ of degree $d=1,2,3$ can be written, respectively,  as weighted hypersurfaces of degree $6$ in $\mathbb{P}(1,1,2,3)$, of degree $4$ in  $\mathbb{P}(1,1,1,2)$, of degree $3$ in $\mathbb{P}^3$.\footnote{To avoid confusion, we emphasize that $d$ denotes the anticanonical degree $K^2_{X_d}$, while $d_1$ denotes the degree of the weighted homogeneous polynomial defining $X_d$ as a weighted hypersurface.}
In other words, we may apply the construction of Section \ref{sec:Fano-wci}, with $c=1$, $m=4$, and
\[
d_1,(a_1, a_2, a_3,a_4)= \begin{cases}
    6,(1,1,2,3) & d=1\\
    4,(1,1,1,2) & d=2\\
    3,(1,1,1,1) & d=3
\end{cases}
\]
In what follows, we consider the three cases together, suppressing the dependence on $d$ from the notation. 

In each case, 
the constant $\alpha$ in \eqref{eq:G-I} is  
$\alpha=\frac{d_{1}!}{a_{3}!a_{4}!}$ (that is, $\alpha=60, 12,6$ for $d=1,2,3$).
Condition \eqref{eq:HV-condition} is satisfied by 
\[d_{1}=a_{2}+a_{3}+a_{4}=1+a_{3}+a_{4}
\]  
thus the HV mirror to $X$ is  
the LG model $(Y^\circ , v^\circ)$ where 
\begin{equation}
Y^\circ=\left( x_1x_2x_3^{a_{3}}x_4^{a_{4}}=1, x_2+x_3+x_4=1\right) \subset (\CC^\times)^4
\end{equation} 
and $v^\circ \colon Y^\circ \to \CC$ is the map $v^\circ(x_1, x_2,x_3, x_4)=x_1$. 
The map $v^\circ$ has image  $\CC^\times \subset \CC$, and has a unique non degenerate critical 
point over the critical value 
$\lambda_0=
\frac{d_{1}^{d_{1}}}{{a_{3}}^{ a_{3}}  {a_{4}}^{a_{4}}}
$
(that is, $\lambda_0 =432, 64, 27$ for $d=1,2,3$).
The fibre of $v^\circ$ over $\lambda$ is the curve $Y^\circ_{\lambda}$:
\begin{equation}
\label{eq:HV-fiber}
\left( \lambda (1-x-y) x^{a_3} y^{a_4}-1=0 \right) \subset (\CC^\times)^2
\end{equation} 
with $x,y$ coordinates on $(\CC^\times)^2$.
For $\lambda  \neq \lambda_0$ the curve $Y^\circ_{\lambda}$ is a smooth curve of geometric genus $1$.

The Przyjalkowski method applied to the HV mirror above gives the Laurent polynomial
\begin{equation*}
    \label{eq:P-method-dP123}
g=\frac{(1+y_3+y_4)^{d_1}}{y_3^{a_3} y_4^{a_4}} 
\end{equation*}
The paper \cite[Figure 1]{Cor15}\footnote{Observe the typo in the bottom right polygon ($S_2^2$). The coefficients of the monomials corresponding to the lattice points in the interior should be equal to $8$.} exhibits a representative Laurent polynomial mirror $f$ to $X$. One can check that Laurent polynomials $g-\alpha$ and $f$ are mutation-equivalent.

\subsection{The associated elliptic surfaces}
\label{sec:elliptic-surf}

The mirrors to del Pezzo surfaces described above naturally give rise to rational minimal elliptic surfaces with section. After recalling a few basic notions about elliptic surfaces, we describe the specific elliptic surfaces that arise in our cases of interest.  

\subsubsection{Basic notions} 
\label{sec:basic-notions}
%include smooth, minimal, section %references only for facts
A smooth projective surface $S$ is an \emph{elliptic surface} if it supports a genus-$1$ fibration $f$ to a smooth projective curve $C$ (i.e. 
a surjective morphism $f \colon S \to C$ such that almost all fibres are smooth curves of genus 1). We say that 
$S$ is an elliptic surface over $C$, or
that $f \colon S \to C$ is an elliptic surface if we wish to specify 
that $C$ is the base curve, or 
that $f$ is the map.  
An elliptic surface $f\colon S \to C$ is \emph{minimal} if no fibre of $f$ contains a $(-1)$-curve. Given an elliptic surface $f \colon S \to C$, one can obtain a minimal one by contracting all $(-1)$-curves contained in fibres. This minimal model is unique up to isomorphism \cite[II.1]{Miranda-ES}. 
We say that  $f\colon S \to C$ is an \emph{elliptic surface with section} if a section of $f$ is given. 
The possible fibre types of a minimal elliptic surface with section are listed in \cite[1.4.1]{Miranda-ES}.

Elliptic surfaces with section admit a theory of Weierstass forms. 
Let $X$ be a projective surface and let $C$ be a smooth projective curve. A \emph{Weierstrass fibration} $\pi \colon X \to C$ is a flat proper map such that every fiber has arithmetic genus one, 
the general fibre is smooth, and a section of $\pi$ not passing through the singularities of any fibre is given. 
Isomorphism classes of Weierstrass fibrations over $C$ are in bijection with \emph{Weierstrass data} over $C$, that is triples $(L,A, B)$ where $L$ is a line bundle on $C$ and $A,B$ are global sections of $L^4, L^6$ such that $4A^3+27B^2$ (the discriminant) is not identically zero, up to isomorphism \cite[II.5]{Miranda-ES}.

There is an injection $F$ from minimal elliptic surfaces with section 
to Weierstrass fibrations (both considered up to isomorphism), given by contracting all components of fibers not meeting the section. There is also a surjection $G$ in the opposite direction, given by
taking the minimal resolution of the Weierstrass fibration.
While $G \circ F=\mathrm{id}$, there are plenty of Weierstrass fibrations not hit by $F$ \cite[II.3, III.4]{Miranda-ES}. A Weierstrass fibration is \emph{in minimal form} if it is hit by $F$.
One naturally wants to consider Weierstrass fibrations in minimal form, as they are in bijection with minimal elliptic surfaces with section. Weierstrass fibrations in minimal form correspond to Weierstrass data $(L,A,B)$ such that for all $c \in C$ $v_c(A)<4$ or $v_c(B)<6$, where $v_c(s)$ denotes the order of vanishing of $s$ at $c$ \cite[III.3.2]{Miranda-ES}, called Weierstrass data in minimal form. 
The fibre types of a minimal elliptic surface with section can be read off from the corresponding Weierstrass data in minimal form via Tate's algorithm, for which we refer the reader to \cite[Section 4.2]{SS}. 
\smallskip

In this paper we consider rational elliptic surfaces. If $S$ is rational, then $C\simeq \pr 1$ \cite[III.4.1]{Miranda-ES}. 
Given Weierstrass data $(\mathcal{O}_{\pr 1}(n),A,B)$ in minimal form,  dehomogenising $A,B$ 
yields a Weierstrass form 
\begin{equation} 
\label{eq:gm-wf}
y^2=x^3+a(\lambda)x+b(\lambda)
\end{equation} where $\lambda$ is a parameter on $\CC \subset \pr 1$, such that for all $c \in \CC$, $v_c(a)<4$ or $v_c(b)<6$. We call such a Weierstrass form \emph{globally minimal}. For a globally minimal Weierstrass form, let $n$ be the smallest integer such that $\deg a \leq 4n$ and $\deg b \leq 6n$. 
Homogenising $a,b$ to two-variable polynomials $A,B$ of degree $4n, 6n$ yields Weierstrass data $(\mathcal{O}_{\pr 1}(n),A,B)$ in minimal form\footnote{The term \emph{globally minimal} is suggestive of the fact that the corresponding Weierstrass data are automatically minimal at $\infty$, that is, $v_\infty(A)<4$ or $v_\infty(B)<6$.}.
Altogether, this defines a bijection between minimal ellitpic surfaces over $\pr 1$ with section and globally minimal Weierstrass forms.

A minimal elliptic surface $f \colon S \to \pr 1$ with section is rational if and only if $n=1$ \cite[III.4.6]{Miranda-ES}. In this case it is easy to see that a Weierstrass form is globally minimal if and only if the following two conditions hold:
\begin{enumerate}
    \item[(i)] $\deg a \leq 4, \deg b \leq 6$, and
    \item [(ii)] the discriminant $\Delta=4a^3+27b^2$ is not a twelfth power.
\end{enumerate}

\subsubsection{Construction}
\label{sec:construction-es}
One way to build the elliptic surface
is as follows. 
Let $X$ be a del Pezzo surface and let $f$ be a mirror Laurent polynomial to $X$. 
Let $P=\mathrm{Newt}(f)$, and let $Y_{P}$ be the toric variety with fan the normal fan of $P$. The polygon $P$ contains the origin  (this is a defining property of maximally mutable Laurent polynomials  (MMLPs) \cite[Definition 2.5]{CKPT}). 
Hence, both $f$ and the unit monomial determine sections of the
polarisation
of $Y_{P}$ defined by $P$ \cite[(4.2.7)]{cox2011toric}. One can check that resolving the singularities of $Y_{P}$ and the base locus of the rational map $Y_{P} \dasharrow \mathbb{P}^1$ defined by $[1:f]$
determines a rational elliptic surface with section. This also follows from a general result by Tveiten \cite[Theorem 3.13]{Tve18} about MMLps in dimension two  (see also \cite[Theorem 3.13]{CKPT}). 
Finally, one considers the corresponding  minimal elliptic surface.

\begin{remark}
If $X$ has very ample anticanonical divisor (that is, it has degree $d \geq 3$), then $P$ is a reflexive polygon, the corresponding polarisation of $Y_P$ is the toric boundary, and,  
 for a general $\mu \in \CC$, the closure of $(f=\mu)$ in $ Y_P$ is a smooth genus-one curve. In this context, the procedure described above is carried out in \cite{GGLP}. 
 For del Pezzo surfaces of degree $d=1,2$, the properties just mentioned do not hold.
 However, the procedure still applies. 
\end{remark}

For del Pezzo surfaces $X_d$ of degree $d=1,2,3$, we denote by $F_d \colon S_d \to \mathbb{P}^1$ the associated minimal elliptic surface. From now on, we will omit the subscript $d$ when the meaning is clear from the context.

The map $F$ has three singular fibres, over $\mu=-\alpha, \lambda_{0}-\alpha,\infty$ (with $\lambda_{0}, \alpha$ as in Section \ref{sec:B-nonproper}), whose type is specified in \eqref{eq:fibres}.
The elliptic surface  $F \colon S \to \mathbb{P}^1$ is extremal (that is, its Mordell-Weil lattice has rank zero), thus it is unique up to isomorphism \cite[Section 5]{miranda_persson}.

Note that the corresponding Weierstrass fibration in minimal form has a singularity of type $E_{\ell}$, where $\ell=9-d$, lying over $\mu=-\alpha$. 

\begin{remark}
The three elliptic surfaces we consider appear in \cite{KeatingCusp} 
as mirrors to the Milnor fibres of the three simple elliptic singularities.
In this work, a slightly different construction of $F \colon S \to \mathbb{P}^1$ is given. 
The result in \cite{KeatingCusp} is the categorical enhancement of the duality between $X$ and a mirror Laurent polynomial. 
\end{remark}

In what follows 
we exhibit a globally minimal Weierstrass form  for $F \colon S  \to \mathbb{P}^1$ about $\mu=-\alpha$ 
(that is, in \eqref{eq:gm-wf}, $\lambda=\mu+\alpha$). 
We extract the Weierstrass form from the HV mirror: we leverage the fact that the Newton polygon associated to the curve $Y^\circ_\lambda$ in \eqref{eq:HV-fiber}, which we denote by $Q$, has width $2$, allowing to present $Y^\circ_\lambda$ as a 2--1 branched cover of $\CC^\times$. A
similar computation is carried out with a different motivation in \cite[Section 3.2]{GMV} -- the
three elliptic surfaces arising as cases 1–3.\footnote{A Weierstrass form  for $F \colon S \to \mathbb{P}^1$ can also be read off from \cite[Table 5.2]{miranda_persson}, but we include the computation above to derive it directly from the mirror, as part of our goal to carry out the constructions explicitly.
} 

Write $M=\ZZ^2$, and let $Q \subset M_\mathbb{R}= \mathbb{R}^2$. %Let $\lambda\neq \lambda_0$. 
In each case, $Q$ has width $2$ with respect to the vector $(-1,1) \in M^\vee$. 
Choosing as basis for $M$ the dual vectors to $(1,0), (-1,1) \in M^\vee$, 
equation \eqref{eq:HV-fiber} rewrites as 
\begin{equation}
\label{eq:y-deg-2}
\lambda\left(1-\frac{x}{y}-y\right)\left(\frac{x}{y}\right)^{a_3}y^{a_4}-1=0
\end{equation}
Up to multiplying by $y$ when $d=3$, the left-hand side of \eqref{eq:y-deg-2} defines a degree-two polynomial in $\mathbb{C}(\lambda, x)[y]$. 
Take the discriminant of this polynomial (divided by $x^2$ when $d=1$) and consider the quadratic equation it defines:
\begin{equation}
\label{eq:y2-disc}
\begin{split}
    y^2&=-4 \lambda^2 x^3+\lambda^2 x^2 -4 \lambda \quad \  \ \text{for} \  d=1 \\
    y^2&= -4 \lambda^2 x^3+\lambda^2 x^2 -4\lambda x \quad \text{for} \ d=2 \\
    y^2&=-4 \lambda^2 x^3 +(-\lambda x+1)^2 \quad \text{for} \ d=3    
\end{split}
\end{equation}
The affine curves cut out of $\mathbb{C}^\times  \times \mathbb{C}$ (with variables $x,y$) by \eqref{eq:y-deg-2} and \eqref{eq:y2-disc} are isomorphic. Moreover, the closure of the affine curve given by \eqref{eq:y2-disc}  in $\mathbb{P}^2_{\mathbb{C}(\lambda)}$ is smooth. 
By a small manipulation 
(absorb the coefficient $-4\lambda^2$ in a new homogeneous variable, then compress the cubic),
\begin{comment}
Ok, quindi se capisco:

Prendiamo il discriminante, che in generale dà solo birazionale ma nel nostro caso ci dà proprio che l aperto su cui sono isomorfe è tutto tranne potenzialmente il bordo.

Quindi se la curva da cui partiamo era liscia fuori dal bordo, lo è anche quella discriminante.

Poi si controlla che la discriminante è liscia pure al bordo (sempre per lambda non lambda0), e allora la chiusura del discriminante makes us happy
\end{comment}
we obtain the  Weierstrass forms
\begin{equation}
\label{eq:Weierstrass-eqn}
\begin{split}
  y^2&=x^3 - \frac{1}{3}\lambda^4 x + \frac{2}{27}\lambda^6- 64\lambda^5  \qquad \qquad  \qquad \qquad \quad \ \text{for} \  d=1\\
y^2&=x^3 + \left(-\frac{1}{3}\lambda^4 + 16\lambda^3\right)x + \frac{2}{27}\lambda^6 - \frac{16}{3}\lambda^5  \qquad \quad  \ \text{for} \  d=2\\
y^2&=x^3 + \left(-\frac{1}{3}\lambda^4 + 8\lambda^3\right)x + \frac{2}{27}\lambda^6 - \frac{8}{3}\lambda^5 + 16\lambda^4 \quad  \text{for} \  d=3
  \end{split}
\end{equation}
\smallskip

In what follows we will write $p(\lambda, x)= x^3 +a(\lambda) x +b(\lambda)$ for the right hand side of \eqref{eq:Weierstrass-eqn}. Note that the discriminant $\Delta$ is given by  
\begin{equation}
\label{eq:discriminant}    \Delta(\lambda) = \lambda^{\ell+2}(\lambda - \lambda_0) \end{equation}
up to constants.

\begin{lemma}
\label{lem:Wf1}    Equations \eqref{eq:Weierstrass-eqn} yield globally minimal Weierstrass forms for the elliptic surfaces $F_d\colon S_d \to \pr 1$.
\end{lemma}

\begin{proof} 
Since the degrees of ${a}$ and ${b}$ do not exceed $4$ and $6$ and ${\Delta}$ is not a twelfth power (see conditions (i) and (ii) in Section \ref{sec:basic-notions}),  $y^2= {p}(\lambda,x)$ is a globally minimal Weierstrass form for a rational minimal elliptic surface.
By \eqref{eq:discriminant}, the singular fibres lie over $0, \lambda_0, \infty$. The fibre types are determined by Tate's algorithm.  
Since $\lambda_0$ is a simple root of $\Delta$, %and $a$ and $b$ do not vanish at $\lambda_0$, 
the fibre over $\lambda_0$ is of type $I_1$. Since $\deg{a}=4$, $\deg {b}=6$ and $\deg
{\Delta}=\ell+3$, the fibre over $\infty$ is of type $I_{d}$. 
For $d=1$, $v_0(a)=4$, $v_0(b)=5$, thus the fibre over $0$ is of type $\mathrm{II}^\star$. The other two cases are handled analogously. 
Then, the statement follows directly from uniqueness of $F_d \colon S_d \to \pr 1$.
\end{proof}

In what follows, we will consider the LG model $(Y,w)$ defined by:
\begin{equation}
\label{eq:Yw}
Y\coloneqq S \setminus F^{-1}(\infty) \qquad w\coloneqq F_{|Y}
\end{equation}
The LG model $(Y,w)$ is a cohomological mirror to $X$ with proper superpotential. 
Note that $w$ is not a Lefschetz fibration, %not isolated,  not degenerate
as the fibre of $w$ over $\mu=-\alpha$, which is non-reduced,
contains non-isolated critical points of $w$  (see the next Section \ref{sec:ConstructionLefschetzFibrations}).

\begin{remark}
\label{rem:jk}
Del Pezzo surfaces of degree $2$ are the degenerate case $(k=0)$ of the series of surfaces $X_{8k+4} \subset \mathbb{P}(2,2k+1,2k+1,4k+1)$ ($k>0$) \cite{JK}.  For these surfaces,  \cite{ACGG} constructs the only known cohomological mirrors $(Y_k, w_k)$ of dimension 2.
For $k=0$, the LG model $(Y_k,w_k)$ agrees with the standard cohomological mirror to $X_2$: the map $w_0$ is (after translation of the base by $\alpha_2$) the restriction over $\CC^\times$ of  $F_2 \colon S_2 \to \mathbb{P}^1$. 
\end{remark}

\section{Construction of Lefschetz fibrations}
\label{sec:ConstructionLefschetzFibrations}
In this Section, for $d=1,2,3$, we construct a rational minimal elliptic surface $\widetilde{F}_d \colon \widetilde{S}_d \to \pr 1$ whose fibres over $\CC$ are all of type $I_1$ and
whose fibre type over $\infty$ matches that of $F_d$. We then consider the LG model $(\widetilde{Y}_d, \widetilde{w}_d)$ obtained by restricting the fibration to $\CC$. We compute an associated collection of vanishing cycles, and determine their homology classes with respect to a fixed basis.  

In Section \ref{sec:geometry-cats}, we show that  $(\widetilde{Y}_d, \widetilde{w}_d)$ is diffeomeorphic to the categorical mirror to $X_d$ in \cite{AKO06}, defined by restricting to $\CC$  
an elliptic surface with the same fiber configuration as $\widetilde{F}_d$. This implies that $(\widetilde{Y}_d, \widetilde{w}_d)$ is a categorical mirror.

\subsection{Preliminaries on Lefschetz fibrations}\label{sec:prelim-Lf}

We begin by recalling some preliminary concepts.

\subsubsection{Vanishing cycles and Lefschetz thimbles} 
\label{sec:vc-th}
Let $M$ be a $\cC^\infty$ manifold of (real) dimension $2n$, and let $C$ be the Riemann sphere $S^2 \simeq \pr 1$ or an open subset of  $S^2$.  Let $w \colon M \to C$ be a \emph{Lefschetz fibration}, i.e. a  $\cC^\infty$ map with isolated non-degenerate critical points.

Let  $\lambda_\ast$ be a regular value of $w$, and write $\Sigma \coloneqq w^{-1}(\lambda_*)$. Let $p$ be a critical point of $w$, and let $\gamma \subset \CC$ be an arc joining $\lambda_*$ to the critical value $w(p)$. Using parallel transport (with respect to a chosen horizontal distribution) along $\gamma$, one obtains:
\begin{itemize}
\item a \textit{vanishing cycle} $L \subset \Sigma$, consisting of the set of points in $\Sigma$ which collapses to $p$ through parallel transport along $\gamma$. 
An isotopic change of parallel transport yields isotopic vanishing cycles, thus a well-defined homology class in $H_{n-1}(\Sigma,\Z)$ (up to a choice of orientation).
\item a \textit{Lefschetz thimble} $D$, consisting of  the set of points swept-out by $L$ under parallel transport along $\gamma$. By construction, $D$ is fibered above $\gamma$ and $\partial D=L$.
\end{itemize}

When the fibers of $w$ are not compact, parallel transport is not always well-defined. 
However, from now on, we always work with proper Lefschetz fibrations.

Denote by $p_1, \dots, p_r$ the critical points of $w$ and, for $i=1, \dots, r$, let $\lambda_i\coloneqq w(p_i)$. From now on we assume for simplicity that the %critical values 
$\lambda_i$ are distinct.\footnote{If this is not the case, one can reduce to this situation by interpreting a critical value with $l$ critical points over it as $l$ distinct critical values that are infinitesimally close, and choosing the corresponding arcs as small deformations of one another. In the proof of Theorem \ref{thm:MutationSequence} we will consider Lefschetz fibrations over $\pr 1$ with $d$  critical points over $\infty$ ($d=2,3$), and use this approach implicitly. }
We choose a collection of arcs $\gamma_1, \dots, \gamma_r$, such that, for $i=1,\ldots,r$, $\gamma_i$ joins $\lambda_*$ to  $\lambda_i$,  the $\gamma_i$ intersect only at $\lambda_*$, and they are ordered clockwise. For each $i$, we denote by $L_i, D_i$ the vanishing cycle and the Lefschetz thimble determined by $\gamma_i$. 
After a small perturbation, we can always assume that the 
intersections among the $L_i$ inside $\Sigma$ are transverse.
We write $k \coloneqq n-1$, and we denote by $l_i$ the homology class of $L_i$ in $H_{k}(\Sigma,\Z)$.

\subsubsection{Picard-Lefschetz formula}
\label{ssec:PicardLefschetz}
Associated to the fibration $w \colon M \to C$ is its monodromy representation 
\[
\rho\colon \pi_1(U,\lambda_*) \to \Aut(H_{k}(\Sigma, \CC))
\]
where $U\subset C$ is the regular locus of $w$.

For each $i$, let $\sigma_i$ be a small circle around $\lambda_i$ containing no other critical value, and define $\chi_i$ to be an element of $\pi_1(U,\lambda_*)$ obtained by following the arc $\gamma_i$ from $\lambda_*$ to its point of intersection with $\sigma_i$, going around $\sigma_i$ once in a clockwise direction, then back along $\gamma_i$. The local monodromy operator $\rho(\chi_i)$ is given by the Picard--Lefschetz formula
\begin{equation}\label{eq:PicardLefschetzGeneral}
\rho(\chi_i)(v)= v + (-1)^{\frac{(k+1)(k+2)}{2}}\pair{v,l_i}l_i 
\end{equation}
where $\pairing$ is the intersection pairing on $H_{k}(\Sigma, \ZZ)$.

When the fibres of $w$ have (real) dimension $2$, that is, $k=1$,  
the formula above becomes 
$\rho(\chi_i)(v)= v + \pair{l_i,v}l_i$
(where the intersection pairing $\pairing$ is so that
$\pair{u,v}=1$ if $u$ and $v$ intersect once in a clockwise direction).  
Note that, in a neighbourhood of $L_i$, $\rho(\chi_i)$ 
is the action on homology of a left Dehn twist with respect to $L_i$, which we denote $\tau^{-1}_{L_i}$ in agreement with other sources in the literature (e.g. \cite{Carlson}).\footnote{
In \cite[\S~4.2]{Carlson} 
the arcs $\gamma_i$ are ordered counterclocwise, 
one goes around $\sigma_i$ counterclockwise, and the intersection product $\langle \cdot, \cdot \rangle$ is opposite of ours. Taking this into account, the formula in \cite[Theorem 4.2.1]{Carlson} agrees with \eqref{eq:PicardLefschetzGeneral}.}
\smallskip

Formula \eqref{eq:PicardLefschetzGeneral} determines the (homology classes of the) vanishing cycles associated to a new ordered collection of arcs. More precisely, suppose that we swap
the order of two critical points. 
For example, swap $p_i$ and $p_{i+1}$, and modify $(\gamma_i,\gamma_{i+1})$ to  $(\gamma_i'=\gamma_{i+1},\gamma_{i+1}')$ or $(\gamma_{i}'',\gamma_{i+1}''=\gamma_i)$ as in Figure \ref{fig:OrderChange}.
Then the vanishing cycles $(L_i',L_{i+1}')$ corresponding to $(\gamma_i',\gamma_{i+1}')$ satisfy 
\[ L_{i}'= L_{i+1}, \qquad \quad   L_{i+1}'= \tau_{L_{i+1}}^{-1}(L_i) \]
and similarly the vanishing cycles $(L_i'',L_{i+1}'')$ associated to $(\gamma_i'',\gamma_{i+1}'')$ satisfy
\[ L_{i}''= \tau_{L_i} (L_{i+1}) \qquad\qquad L_{i+1}''=L_i \]

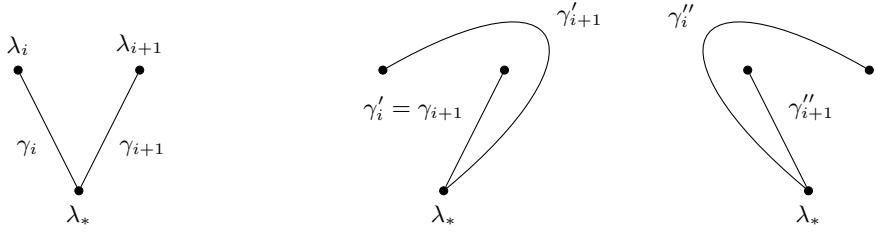
\begin{figure}[ht!]
\centering
\begin{tikzpicture}[scale=0.8, font=\footnotesize]
\node (*) at (0,0) [dot, label=below:{$\lambda_*$}]{};
\node (A) at (-1,2) [dot, label=above:{$\lambda_i$}]{};
\node (B) at (1,2) [dot, label=above:{$\lambda_{i+1}$}]{};

\draw (*) edge["{$\gamma_i$}"] (A);
\draw (*) edge[swap, "{$\gamma_{i+1}$}"] (B);

\node (*) at (6,0) [dot, label=below:{$\lambda_*$}]{};
\node (C) at (5,2) [dot, label=above:{}]{};
\node (D) at (7,2) [dot, label=above:{}]{};

\draw (*) edge[swap, "{$\gamma_{i+1}'$}", out=40, in=30, looseness=4] (C);
\draw (*) edge["{$\gamma_{i}'=\gamma_{i+1}$}"] (D);

\node (*) at (12,0) [dot, label=below:{$\lambda_*$}]{};
\node (C) at (11,2) [dot, label=above:{}]{};
\node (D) at (13,2) [dot, label=above:{}]{};

\draw (*) edge[swap,"{$\gamma_{i+1}''$}"] (C);
\draw (*) edge["{$\gamma_i''$}", out=140, in=150, looseness=4] (D);
\end{tikzpicture}
\caption{An example %where the order of two critical values is swapped. 
of a choice of $(\gamma_i',\gamma_{i+1}')$ and $(\gamma_i'',\gamma_{i+1}'')$ which swaps the order of two critical values.
}
\label{fig:OrderChange}
\end{figure}

We say that the collections of vanishing cycles 
\begin{equation*}
\begin{alignedat}{2}
    (L_1,\ldots, L_{i-1},& L_i,L_{i+1}, L_{i+2}\ldots, L_r)\\
    (L_1,\ldots, L_{i-1},& L_{i+1}, \tau_{L_{i+1}}^{-1}(L_i) , L_{i+2}, \ldots, L_r)
\end{alignedat}
\end{equation*}
are related by a right mutation at the $i$-th position. Similarly, the collections
\begin{equation*}
\begin{alignedat}{2}
    (L_1,\ldots, L_{i-1},& L_i,L_{i+1}, L_{i+2} \ldots, L_r)\\
    (L_1,\ldots, L_{i-1},& \tau_{L_i} (L_{i+1}), L_i, L_{i+2}, \ldots, L_r)
\end{alignedat}
\end{equation*}
are related by a left mutation at the $i$-th position.

Every collection $\{\gamma_i'\}$ can be obtained from $\{\gamma_i\}$ by a finite sequence of the moves described above \cite{Seidel_VanishingCycles}, %\cite[Sec. 2.5]{GHS13:matter},
thus the corresponding collections of vanishing cycles are related by a finite sequence of right or left mutations. %???

\subsubsection{Symplectic Lefschetz fibrations}
\label{sec:symplectic-Lf}
If $(M,\omega)$ is a symplectic manifold, a Lefschetz fibration $w \colon M \to C$ is called \textit{symplectic} if the fibers of $w$ are symplectic submanifolds of $M$.  
If $w \colon M \to C$ is a symplectic Lefschetz fibration, then there is a canonical horizontal distribution defined by the symplectic orthogonal to the fibers of $w$. This allows us to consider the $L_i$ as Lagrangian submanifolds of $\Sigma$, well-defined up to Hamiltonian isotopy. 
The symplectic structure is crucial in the definition of the (bounded) derived category of Lagrangian vanishing cycles. 

Recall that the directed category of Lagrangian vanishing cycles $\lagvc(w,\{ \gamma_i\})$ is an $A_\infty$-category (over a coefficient ring $R$) with objects $L_1,\ldots, L_r$ corresponding bijectively to the vanishing cycles (%which correspond bijectively to 
or to the Lefschetz thimbles). Morphisms are given by: % the Floer complex
\begin{equation}
    \Hom^*(L_i,L_j) = \begin{cases} CF^*(L_i,L_j; R) &\mbox{ if }i<j;\\
    \id[0]  &\mbox{ if }i=j;\\
    0 &\mbox{ otherwise,}
    \end{cases}
\end{equation}
where $ CF^*(L_i,L_j; R)=R^{|L_i\cap L_j|}$ is the Floer complex of $L_i$ and $L_j$,  and the operators $m_k$ ($k\geq 1$) are defined in terms of Lagrangian Floer homology within $w^{-1}(\lambda_*)$.  For the details of the construction, as well as aspects on the grading, we direct the reader to \cite{Seidel_VanishingCycles} and \cite[Section 4]{AKO06}. 

The category $\lagvc(w,\{\gamma_i\})$ depends on the ordered collection of arcs $\gamma_i$. A different collection $\{\gamma_i'\}$ produces a category which differs from $\lagvc(w,\{\gamma_i\})$ by a sequence of mutations (see Section \ref{ssec:PicardLefschetz}). Hence, the derived category $D^b\lagvc(w)$ only depends on the symplectic fibration $w$ and is independent of the choice of the ordered collection of arcs.
\smallskip

In this paper we will mostly focus on the topological data, which, in the case of genus-$1$ Lefschetz fibrations (that is, Lefschetz fibrations whose smooth fibers are $2$-tori), is sufficient to recover the 
Grothendieck group $K_0(D^b\lagvc(w))$.

In this situation, the relative homology group $H_2(M,\Sigma;\Z)$ is equipped with a bilinear form $\pairing_{\mathrm{Sft}}$, called the Seifert pairing in \cite[Definition 4.2]{HarderThompson20}, after a similar notion in singularity theory \cite[Section 2.3]{AGZVII}. 
By \cite[Theorem 2.5]{GHS16_junctions} $H_2(M,\Sigma;\Z)$ is generated by homology classes of Lefschetz thimbles, and the Seifert pairing can be written conveniently in terms of the intersection product on a reference fiber \cite[Section 4.2]{HarderThompson20}.  More explicitly, pick Lefschetz thimbles $D_i$ which give rise to a basis of $H_2(M,\Sigma;\Z)$. Then
\begin{equation}\label{eq:SeifertPairing}
    \pair{[D_i],[D_j]}_{\mathrm{Sft}} = \begin{cases}
        \pair{[\partial D_i],[\partial D_j]} &\mbox{for } i<j\\
        1 &\mbox{if }i=j\\
        0 &\mbox{otherwise,}
    \end{cases}
\end{equation} where $[D]$ and $[\partial D]$ denote the classes of $D$ and $\partial D$ in $H_2(M,\Sigma;\Z)$ and $H_1(\Sigma, \ZZ)$. 
The pairing extends by linearity to $H_2(M,\Sigma;\Z)$.

Since $D^b\lagvc(w)$ admits full exceptional collections, its Grothendieck group coincides with the numerical one. Moreover, it is simple to see that $K_0(D^b\lagvc(w))$, with the Euler pairing, is isometric to $H_2(M,\Sigma;\Z)$ with the Seifert pairing. We abbreviate $K_0(D^b\lagvc(w))$ by $K_0(w).$

\begin{remark} 
The pair $(H_2(M,\Sigma;\Z),\pairing_{\mathrm{Sft}})$ is a \textit{pseudolattice}, i.e. a finitely generated abelian group with a non-degenerate bilinear (not necessarily symmetric) form. We give more detail about pseudolattices in Section \ref{sec:PseudolatticesAnd Mutations}. 
\end{remark}

\subsection{Construction of Lefschetz fibrations}
\label{sec:construct-Lf} 
We now return to the elliptic surfaces $F \colon S \to \pr 1$ and the LG models $(Y,w)$ introduced in Section \ref{sec:construction-es} (continuing to omit the subscript $d$ when it is clear from the context).
As previously noted, the LG models $({Y},{w})$ do not define Lefschetz fibrations. 
We construct associated LG models  $(\widetilde{Y}, \widetilde{w})$ that define Lefschetz fibrations by perturbing Equations \eqref{eq:Weierstrass-eqn} as follows. 
\smallskip

As before, write $\ell \coloneqq 9-d$, and let $p(\lambda, x)= x^3 +a(\lambda) x +b(\lambda)$ be the right hand side of \eqref{eq:Weierstrass-eqn}. 
The critical values of $w$ are the zeros of the discriminant $\Delta$ in \eqref{eq:discriminant}. 
Consider a perturbation \[\widetilde{p}(\lambda,x)=x^3 +\widetilde{a}(\lambda) x +\widetilde{b}(\lambda)\]  of $p(\lambda, x)$ where $\widetilde{a}, \widetilde{b} \in \CC[\lambda]$ satisfy $\deg\widetilde{a}=\deg a, \deg \widetilde{b}=\deg b$, and the corresponding discriminant $\widetilde{\Delta}$ is a separable polynomial of the same degree as $\Delta$\footnote{In physics language, this operation \textit{Higgses the gauge group}, by making the roots of $\widetilde{\Delta}$ near $0$ have discrete symmetries, see \cite[Equation~(1.1)]{GHS13:matter}.}.
Then:

\begin{lemma}
\label{lem:lemma} 
The equation $y^2= \widetilde{p}(\lambda,x)$ is a globally minimal Weierstrass form for a rational minimal elliptic surface $\widetilde{F}\colon \widetilde{S}\to \pr 1$ with $\ell + 3$ fibres of type $I_1$ over $\C$ and a fibre  of type $I_d$ at $\infty$.
\end{lemma}

\begin{proof}
The proof is almost identical to that of Lemma \ref{lem:Wf1}. 
\end{proof}

We define $(\widetilde{Y}, \widetilde{w})$ as the LG model 
\begin{equation}
\label{eq:TildeYw}
\widetilde{Y} \coloneqq \widetilde{S} \setminus \widetilde{F}^{-1}(\infty) \qquad \widetilde{w} \coloneqq \widetilde{F}_{|\widetilde{Y}}
\end{equation}
By construction, the map  $\widetilde{w} \colon \widetilde{Y} \to \CC$ is a Lefschetz fibration. More precisely, it is a genus-one proper Lefschetz fibration with $\ell+3$ fibres with a single node.  

The explicit choice we will work with is: 
\begin{equation}
     \label{eq:dP-def} 
    \widetilde{p}(\lambda,x) =  p(\lambda,x) + \epsilon x 
\end{equation}
for a fixed $0 < |\epsilon| \ll 1$. The corresponding discriminant $\widetilde{\Delta}$ has $\ell+2$ roots close to $0$ and one root close to $\lambda_0$.
\smallskip

We end the section with a few remarks on our construction and its relation to \cite{AKO06}. 

\begin{remark}
\label{rem:necessary}
It is crucial to start from a globally minimal Weierstrass form in order to produce, via a perturbation 
as above, a rational elliptic surface:
a Weierstrass form that is not globally minimal would perturb to a globally minimal Weierstrass form whose associated elliptic surface is not rational.
Moreover, the conditions on
$\widetilde{p}$ 
  are
  strictly necessary to produce a globally minimal Weierstrass form whose associated elliptic surface is as in Lemma \ref{lem:lemma}. This follows directly from Tate's algorithm. \end{remark}

\begin{remark}
\label{rem:morsification}
The perturbation we perform %in Equation \eqref{eq:dP-def} 
yields an analog of a Morsification of the (non isolated) singularities of $w^{-1}(0)$. We refer to \cite[Section 3.3]{Dimca_SingularitiesAndTopology} for details on Morsifications, and, for example, \cite{KeatingTori} for an application of Morsification techniques.
\end{remark}

\begin{remark}
\label{rem:bifibration}
The map $\widetilde{w}\colon \widetilde{Y}\to \C$ has the structure of a Lefschetz bifibration \cite[III.(15e)]{Seidel}. In other words, we have a factorization
\begin{equation}
\begin{tikzcd}[ampersand replacement=\&]
    \widetilde{Y} \arrow[r,"Q"] \arrow[d,swap,"{\widetilde{w}}"] \& \C \times \pr 1 \arrow[dl,"\mathrm{pr}_1"] \\
      \C  \&
\end{tikzcd}
\end{equation}
where $Q=(\lambda,x)$ with $\lambda$ and $x$ as in Equation \eqref{eq:Weierstrass-eqn} ($x$ is regarded as a variable on $\pr 1$), and $\mathrm{pr}_1$ is projection onto the first factor, such that the restriction of $Q$ to a regular fiber ${\Sigma}=\widetilde{w}^{-1}(\lambda_\star)$ is a Lefschetz fibration. More precisely, ${\Sigma}$ is %(punctured) 
a smooth elliptic curve and ${q}\coloneqq Q_{|{\Sigma}} \colon {\Sigma} \to \pr 1$ is a $2\; \colon1$ cover of $\pr 1$ branched at four distinct points (the three roots $\widetilde{p}(\lambda_\star,x)$, and $\infty$). 
\end{remark}

\subsection{Categorical mirrors, anticanonical pairs and diffeomeorphism of LG models}
\label{sec:geometry-cats}

The paper \cite{AKO06} constructs categorical mirrors $(M_\ell, W_\ell)$ to del Pezzo surfaces  that are blowups of $\pr 2$ in  $\ell$ points, that is all del Pezzo surfaces apart from $\pr 1\times\pr 1$, with the degree $d$ of the surface given by $d=9-\ell$.

The LG model $(M_\ell, W_\ell)$ is defined by restriction to $\CC$ of a rational elliptic surface $\overline{W}_\ell \colon \overline{M} \to \pr 1$ with $\ell+3$ fibres of type $I_1$ over $\CC$ and a fibre of type $I_{9-\ell}$ at $\infty$. Hence, for $d=1,2,3$, $\overline{W}_\ell$ has the same fibre configuration as $\widetilde{F}_d$. 

For each $\ell$, $\overline{W}_\ell \colon \overline{M} \to \pr 1$ is built out of the rational elliptic surface $W_0 \colon \overline{M} \to \pr 1$ extending the Laurent polynomial mirror to $\pr 2$. The map $\overline{W}_0$ has $3$ fibres of type $I_1$ over $\CC$ and a fibre of type $I_9$ over $\infty$, and the map $\overline{W}_\ell$ is obtained by perturbing $\overline{W}_0$ in such a way that 
$\ell$ of the critical points originally over $\infty$
are displaced to 
$\ell$ distinct finite critical values.

We say that two LG models $(N,p)$ and $(N',p')$ are diffeomorphic if there are diffeomorphisms $N \xrightarrow{\sim}N'$ and $\C \xrightarrow{\sim}\C$ intertwining $p$ and $p'$.
Changing the choice of $\overline{W}_\ell$ does not modify the diffeomorphism class of $(M_\ell, W_\ell)$ \cite[Section 3.3]{AKO06}.

We now compare $(M_\ell,W_\ell)$ with $(\Yt_d,\wt_d)$. We will show that, suitably choosing symplectic forms, one has the following.
\begin{proposition}\label{prop:EquivalenceYtildeAko}
There is an equivalence $D^b\lagvc(\wt_d) \simeq D^b\lagvc(W_\ell)$.
\end{proposition}
In particular $(\Yt_d,\wt_d)$ is a categorical mirror to $X_d$, by \cite[Theorem 1.4]{AKO06}.

\begin{proof}
In the terminology of \cite[Remark 4.7]{HarderThompson20}, the two LG models $(\Yt_d, \wt_d)$ and $(M_\ell, W_\ell)$ determine quasi-LG models (with $n=12-d=\ell+3$), hence by \cite[Corollary 4.17]{HarderThompson20} they are diffeomorphic. 
Fixing a suitable symplectic form on $M_\ell$ and its pullback on $\Yt_d$ via the above diffeomeorphism, we obtain the desired equivalence of categories.
\end{proof}

It is also interesting to consider the LG models we study through the lens of anticanonical pairs.

Consider a rational minimal elliptic surface $S$ %with a section 
with a fibre $D$ of type 
$I_{d}$, with $0< d\leq 9$.
The pair $(S,D)$ is an
anticanonical pair of length $d$ in the sense of \cite[Section 1]{Friedman15} where, for $d=1$, $D$ is irreducible and $D^2=0$, while, for $d>1$, each component of $D$ has self-intersection $-2$.
For $d \neq 8$, any two such pairs are deformation equivalent by \cite[Chapter 1]{Loo81} and \cite[Proposition 9.15-16]{Friedman15}.
In our setting, the  elliptic surfaces $F_d \colon S_d \to \pr 1$, $\widetilde{F}_d \colon {S}_d \to \pr 1$, as well as the elliptic surfaces $\overline{W}_{9-d} \colon \overline{M} \to \pr 1$ in \cite{AKO06},
are rational, minimal, and have a fibre of type $I_d$ at $\infty$. Hence, they define deformation equivalent anticanonical pairs of length $d$. 
By Ehresmann's fibration theorem this implies, in particular, that $Y_d$, $\widetilde{Y}_d$, as well as the manifold $M_{9-d}$ in \cite{AKO06}, are diffeomorphic.

We emphasise that the results \cite{Loo81, Friedman15} are non-constructive and rely on lattice-theoretic methods. Moreover, even in the context of our application, the fibres of the deformation are a priori anticanonical pairs rather than elliptic surfaces. An anticanonical pair as above
defines an elliptic surface if and only if $\sO_D(D)\simeq \sO_D$, as in the proof of \cite[Theorem 9.14]{Friedman15}. 

A systematic study of HMS for  anticanonical pairs (also called \textit{log Calabi--Yau surfaces}) 
is conducted in  \cite{KeatingCusp, HackingKeating}.

\begin{example}
    \label{rem:def-Wf-es} 
From the perturbation \eqref{eq:dP-def} one can obtain a deformation of $W\colon S \to \pr 1$ to $\widetilde{W}\colon \widetilde{S} \to \pr 1$, without appealing to the above general theory. More precisely, we can construct a family of elliptic surfaces where $S$ and $\widetilde{S}$ appear as fibers, and where all other fibers have the same singular fiber types as $\widetilde{S}$. This is done as follows.

Let $B$ be a small disc with coordinate $b$. The equation $y^2 = p(\lambda,x) + b x$ gives rise to a family of Weierstrass fibrations $g\colon \sZ \to B$. The surface $\sZ_0 \coloneqq g^{-1}(0)$ is the Weierstrass fibration of $S$. 
A fiber $\sZ_b \coloneqq g^{-1}(b)$ with $b \neq 0$ is the Weierstrass fibration of an elliptic surface with the same singular fiber types as $\widetilde{S}$. In particular, for $b=\epsilon$, $\sZ_b$ is the Weierstrass fibration of $\widetilde{S}$. 

 For $d=2,3$, the threefold $\sZ$ has a line of $A_{d-1}$ singularities, given by the union of the $A_{d-1}$ points over $\infty$ of the fibers $\sZ_\epsilon$. These can be simultaneously resolved by blowups. Let $\sZ^\prime$ be the resulting threefold. For $d=1$, define $\sZ' = \sZ$ (no blowup is needed, as the fibers of $\sZ$ are already smooth everywhere except for the $E_6$ singularity in $\sZ_0$).
For $d=1,2,3$, the fibers of the family $\sZ' \to B$ are all smooth, except for $\sZ_0'$ which has a singular point of type $E_{9-d}$.

By a result of Brieskorn \cite{Brieskorn}, as formulated in \cite[Theorem 4.43]{KM98}, there exists a finite surjective base change $B' \to B$ such that $\sZ'_{B'}\to B'$ admits a simultaneous resolution, i.e. a projective morphism $p:\sS \to \sZ'_{B'}$ such that the composition $h\colon \sS \to \sZ'_{B'} \to B'$ is a flat family over $B'$ and $p$ restricts to the minimal resolution above each $b'\in B'$.
The map $h$ is the desired family. 
Note that for $d=1,2$ a base change is necessary by \cite{Katz} since $\sZ'$ has an $A_4$ singularity.

\end{example}

\begin{remark}\label{rem:Y-mirror}
In order to obtain an equivalence of categories 
$\FS(Y_d,w_d) \simeq D^b\lagvc(\wt_d)$, one could apply the fact that two LG models $(N,p)$ and $(N',p')$ produce equivalent Fukaya--Seidel categories if there exists a diffeomorphism $\psi \colon N \to N'$, two discs $\Delta,\Delta' \subset \C$ containing all critical values of $p$, respectively $p'$, and a diffeomeorphism $\phi \colon \C\setminus \Delta \to \C \setminus \Delta'$ such that $\phi \circ p = p' \circ \psi$ where these maps are defined. 

The deformation of Example \ref{rem:def-Wf-es} yields a diffeomeorphism $\psi \colon Y \xrightarrow{\sim} \Yt$ by Ehresmann's theorem.
%We omit all the necessary details, but i
It is not hard to find a similar deformation interpolating $(Y,W)$ and a Lefschetz fibration equivalent to $(\Yt,\wt)$ (or equivalently to $(M_\ell,W_\ell)$, by Proposition \ref{prop:EquivalenceYtildeAko}) in a way that $\Delta$, $\Delta'$, $\psi$ and $\phi$ as above exist. 
\end{remark}

\subsection{Construction of vanishing cycles}
\label{sec:construct-vc}

Let $\widetilde{w} \colon \widetilde{Y} \to \CC$ be the Lefschetz fibration built in Section \ref{sec:construct-Lf}. 
The group $H_2(\widetilde{Y}, {\Sigma};\Z)$, where ${\Sigma}$ is a smooth fiber of $\widetilde{w}$, comes equipped with the Seifert pairing \eqref{eq:SeifertPairing}. 

In this section, we compute the Gram matrix of the Seifert pairing with respect to a choice of $\Sigma$ and a basis of Lefschetz thimbles $D_i$, proceeding as follows. First, we choose a reference fibre $\Sigma=\widetilde{w}^{-1}(\lambda_\star)$. Next, we order the $\ell+3$ critical values $\lambda_i$ of $\widetilde{w}$, and we choose arcs $\gamma_i$ as described in \S~\ref{sec:vc-th}. Then, we use the fact that $\Sigma$ is a $2\; \colon1$ branched cover of $\pr 1$ (see Remark \ref{rem:bifibration}) to identify the vanishing cycles $L_i$ associated to the arcs $\gamma_i$ as the double lifts of certain arcs $\delta_i \subset \C \subset \pr 1$. Each $\delta_i$ is determined by two roots of $\widetilde{p}(\lambda, x)$ which collide when moving from $\lambda_\star$ to 
$\lambda_i$.\footnote{This is a common procedure when working with bifibrations. It is used for example in \cite[Section 4]{AKO06}.} Finally, we choose a basis of $H_1(\Sigma, \ZZ)$ and we use formula \eqref{eq:SeifertPairing}.
\medskip

We choose as reference fiber $\Sigma \coloneqq \widetilde{w}^{-1}(0)$, which (for each $d$) is determined by the local Weierstass form:
\[ y^2 = x^3 + \epsilon x \]

There is a critical value of $\widetilde{w}$ which is a small perturbation of ${\lambda}_0$. Slightly abusing notation,  we still denote it by ${\lambda}_0$. We order the $\ell+3$ critical values $\lambda_i$ starting from ${\lambda}_0$ and sweeping the plane clockwise (see Figure \ref{fig:critval6} for the case $d=3$). We choose the arcs $\gamma_i$ to be straight-line segments joining $\lambda_\star=0$ and $\lambda_i$.
 
Using computer algebra, we determine the two roots of $\widetilde{p}(\lambda,x)$ that collide when $\lambda$ moves along $\gamma_i$ and plot their trajectory, which defines the arc $\delta_i$. In Figure \ref{fig:homologyOfDeltai}, we choose $\epsilon \in \R_+$ and plot the three roots of $\widetilde{p}(0,x)=x^3+\epsilon x$, given by $0,\pm  \sqrt{-\epsilon}$, along with all the arcs that will arise as one of the $\delta_i$.

We choose a basis of $H_1(\Sigma,\Z)$ as follows. Let $a$ be one of the two (disjoint) cycles in the preimage of a small loop around $0$ and $-\sqrt{-\epsilon}$, and $b$ one of the two cycles in the preimage of a small loop around $0, \sqrt{-\epsilon}$, chosen so that $\pair{a,b}=-1$ (see Figure \ref{fig:Basisab}). By construction we have
\begin{equation}\label{eq:intEllCurve}
     \pair{a,a}=\pair{b,b}=0 \qquad \pair{a,b}=-\pair{b,a}=-1. 
\end{equation}

In Figure \ref{fig:homologyOfDeltai}, we write next to each arc the homology class of the corresponding vanishing cycle with respect to the chosen basis $a,b$.

\begin{figure}[ht!]
\tikzset{every picture/.style={line width=0.75pt}} %set default line width to 0.75pt        
\captionsetup{width=.48\textwidth}
\noindent
\minipage{0.48\textwidth}%
\centering
\begin{tikzpicture}[x=0.75pt,y=0.75pt,yscale=-.7,xscale=.7,font=\footnotesize]

%Shape: Circle [id:dp27823626312354] 
\draw   (76.96,239.21) .. controls (76.83,237.01) and (78.52,235.12) .. (80.73,234.99) .. controls (82.93,234.87) and (84.82,236.56) .. (84.94,238.76) .. controls (85.07,240.97) and (83.38,242.86) .. (81.18,242.98) .. controls (78.97,243.11) and (77.08,241.42) .. (76.96,239.21) -- cycle ;
%Shape: Circle [id:dp8196291652719697] 
\draw   (76.52,181.39) .. controls (76.3,179.19) and (77.91,177.23) .. (80.11,177.02) .. controls (82.31,176.8) and (84.27,178.41) .. (84.48,180.61) .. controls (84.7,182.81) and (83.09,184.77) .. (80.89,184.98) .. controls (78.69,185.2) and (76.73,183.59) .. (76.52,181.39) -- cycle ;
%Shape: Circle [id:dp03955477234120586] 
\draw   (75.68,135.18) .. controls (75.58,132.97) and (77.29,131.1) .. (79.5,131) .. controls (81.71,130.9) and (83.58,132.61) .. (83.68,134.82) .. controls (83.78,137.02) and (82.07,138.89) .. (79.86,138.99) .. controls (77.65,139.09) and (75.78,137.38) .. (75.68,135.18) -- cycle ;
%Straight Lines [id:da7508459467956833] 
\draw  [dash pattern={on 4.5pt off 4.5pt}]  (79.5,31) -- (79.33,64) ;
%Shape: Ellipse [id:dp03805851396772919] 
\draw  [color={rgb, 255:red, 208; green, 24; blue, 2 }  ,draw opacity=1 ] (78.85,264.16) .. controls (66.15,264.16) and (56.12,239.77) .. (56.45,209.69) .. controls (56.77,179.6) and (67.33,155.21) .. (80.03,155.21) .. controls (92.73,155.21) and (102.76,179.6) .. (102.44,209.68) .. controls (102.11,239.77) and (91.55,264.15) .. (78.85,264.16) -- cycle ;
\draw  [color={rgb, 255:red, 208; green, 24; blue, 2 }  ,draw opacity=1 ] (98.12,208.1) -- (102.42,198.54) -- (106.04,208.38) ;
%Curve Lines [id:da572988088937383] 
\draw [color={rgb, 255:red, 74; green, 144; blue, 226 }  ,draw opacity=1 ]   (81.56,196.97) .. controls (118.76,192.1) and (112.37,89.02) .. (79.42,97.5) ;
%Curve Lines [id:da7887447001163369] 
\draw [color={rgb, 255:red, 74; green, 144; blue, 226 }  ,draw opacity=1 ] [dash pattern={on 4.5pt off 4.5pt}]  (81.56,196.97) .. controls (43.5,200) and (45.23,95.68) .. (79.42,97.5) ;
\draw  [color={rgb, 255:red, 74; green, 144; blue, 226 }  ,draw opacity=1 ] (109.34,126.21) -- (106.42,136.28) -- (101.46,127.04) ;
%Straight Lines [id:da10680557715797712] 
\draw    (80.73,184.99) -- (80.73,234.99) ;
%Straight Lines [id:da6064062496018618] 
\draw    (79.33,64) -- (79.5,131) ;

% Text Node
\draw (117,207) node [anchor=north west][inner sep=0.75pt]   [align=left] {$a$};
% Text Node
\draw (119,125) node [anchor=north west][inner sep=0.75pt]   [align=left] {$b$};
\end{tikzpicture}
\caption{  \label{fig:Basisab} The chosen basis for $H_1(\Sigma,\Z)$. The black segments are branch cuts. The $a$ loop lives entirely on one sheet of the cover, while the dashed part of the $b$ loop is meant to be on the second sheet.}
\endminipage\hfill
\minipage{0.48\textwidth}
\centering
\begin{tikzpicture}[font=\footnotesize,scale=0.64]
\node (F+) at (0,2) [dot, label=above:{$\sqrt{-\epsilon}$}]{};
\node (F0) at (0,0) [dot, label=left:{$0$}]{};
\node (F-) at (0,-2) [dot, label=below:{$-\sqrt{-\epsilon}$}]{};

\draw (F-) edge["{$a$}"] (F0);
\draw (F+) edge["{$b$}"] (F0);
\draw (F-) edge[bend left=60, "{$a+b$}"] (F+);
\draw (F-) edge[bend right=60, swap, "{$a-b$}"] (F+);

\end{tikzpicture}
\caption{  \label{fig:homologyOfDeltai} Every arc $\delta_i$ appearing in the computation is isotopic to one of the arcs above. Next to each arc is the homology class of its lift to $\Sigma$.} 
\endminipage
\end{figure}

Following this procedure for each $i$, we find the corresponding arc $\delta_i$ and obtain the following.
\begin{proposition}
\label{pro:vc}
 Up to choosing an orientation, the cohomology classes of the vanishing cycles $L_i$, for $i=0, \dots, \ell+2$, are as follows:
\begin{align}
     \label{eq:vc-dP1}
    &a+b,a,b,a,b,a,b,a,b,a,b & \mbox{for } d=1\\
    \label{eq:vc-dP2}
    &a+b,a,b,a,a,a-b,b,b,a,b &  \mbox{for } d=2\\
    \label{eq:vc-dP3}
&a+b,b,a,b,a,b,a,b,a & \mbox{for } d=3
\end{align}
\end{proposition}
\smallskip

Using \eqref{eq:SeifertPairing} and \eqref{eq:intEllCurve}, one obtains the Gram matrix of the Seifert pairing on $H_2(\widetilde{Y},\Sigma;\Z)$ with respect to the  Lefschetz thimbles corresponding to the $L_i$. We spell out one of these computations for $d=3$ in Example \ref{ex:vc3} below.

\begin{example}[$d=3$]\label{ex:vc3}
Consider $\wt=\wt_3$. There are 9 critical values, arranged as in Figure \ref{fig:critval6}. Consider for example the critical value $\lambda_1$. The roots $0$ and $\sqrt{-\epsilon}$ of $\widetilde{p}$ collide as $\lambda$ varies along $\gamma_1$. Their trajectories define a path $\delta_1$. 
We draw this in Figure \ref{fig:vancycle6}. The lift of $\delta_1$ to the fiber $\Sigma$ has homology class $b$, as in Figure \ref{fig:homologyOfDeltai}.

\begin{figure}[ht!]\tikzset{every picture/.style={line width=0.75pt}} %set default line width to 0.75pt        
\captionsetup{width=.48\textwidth}
\noindent
\minipage{0.48\textwidth}%
\centering
\begin{tikzpicture}[font=\footnotesize,scale=0.8]
\node (N) at (0,3) []{};
\node (E) at (5,0) []{};
\node (S) at (0,-3) []{};
\node (W) at (-3,0) []{};

\draw (S) edge[->] (N);
\draw (W) edge[->] (E);

\node (0) at (4,0) [dot, label=above:{$\lambda_0$}]{};
\node (1) at (2,-1) [dot, label=right:{$\lambda_1$}]{};
\node (2) at (1,-2) [dot,label=right:{$\lambda_2$}]{};
\node (3) at (-1,-2.5) [dot, label=below:{$\lambda_3$}]{};
\node (4) at (-2.5,-0.8) [dot, label=below:{$\lambda_4$}]{};
\node (5) at (-2.5,0.8) [dot, label=above:{$\lambda_5$}]{};
\node (6) at (-1,2.5) [dot, label=above:{$\lambda_6$}]{};
\node (7) at (1,2) [dot, label=right:{$\lambda_7$}]{};
\node (8) at (2,1) [dot, label=right:{$\lambda_8$}]{};

\node (origin) at (0,0) [wdot]{};

\draw (origin) edge[dashed,"{$\gamma_1$}"] (1);

\end{tikzpicture}
\caption{  \label{fig:critval6} The critical values of the fibration $\wt$. In evidence is the path $\gamma_1$ joining the origin to the critical value $\lambda_1$.}
\endminipage\hfill
\minipage{0.48\textwidth}
\centering
\smallskip

\begin{tikzpicture}[font=\footnotesize,scale=0.74]
\node (F+) at (0,2) [dot, label=above:{$\sqrt{-\epsilon}$}]{};
\node (F0) at (0,0) [dot, label=left:{$0$}]{};
\node (F-) at (0,-2) [dot, label=right:{$-\sqrt{-\epsilon}$}]{};
\node (F') at (-0.6,-2.5) []{};

%\draw (F-) edge["{$a$}"] (F0);
\draw (F+) edge[bend left=15, "{$\delta_1$}"] (F0);
\draw (F-) edge[bend right=10] (F');

\end{tikzpicture}
\bigskip

\bigskip

\caption{  \label{fig:vancycle6} Trajectories of the zeros of $\widetilde{p}(\lambda,x)$ as $\lambda$ varies along the path $\gamma_1$ in Figure \ref{fig:critval6}.} 
\endminipage
\end{figure}

The Gram matrix of the Seifert pairing on $H_2(\Yt,\Sigma,\Z)$ in the basis of thimbles corresponding to \eqref{eq:vc-dP3} is
{
\begin{equation}
\label{eq:gram-3}
M =
\begin{pmatrix}
1&-1&1&-1&1&-1&1&-1&1\\
0&1&1&0&1&0&1&0&1\\
0&0&1&-1&0&-1&0&-1&0\\
0&0&0&1&1&0&1&0&1\\
0&0&0&0&1&-1&0&-1&0\\
0&0&0&0&0&1&1&0&1\\
0&0&0&0&0&0&1&-1&0\\
0&0&0&0&0&0&0&1&1\\
0&0&0&0&0&0&0&0&1
\end{pmatrix}
\end{equation}
}
\smallskip

\end{example}

\section{Pseudolattices and Mutations}\label{sec:PseudolatticesAnd Mutations}

In this section we give explicit sequences of (numerical) mutations which map the exceptional bases obtained in \ref{sec:construct-vc} to those corresponding to the exceptional collections \cite[Equation 2.3]{AKO06}. This is the content of Theorem \ref{thm:MutationSequence}. We give an application to the representation theory of string junctions in Section \ref{sec:GHS}. 
We use the language of pseudolattices, which we briefly summarize below.

\subsection{Surface-like pseudolattices}
\label{sec:surfacelikePseudolatt-prelim}

For the main results of the theory of surface-like pseudolattices, we direct the reader to \cite{VdB_dTdV,Kuz17_pseudolattices,Vial17}.

\begin{definition}
    A pseudolattice is a finitely generated abelian group $\rG$ equipped with a non-degenerate bilinear form $\pairing_\rG \colon \rG \times \rG \to \Z$. We say that $(\rG,\pairing_\rG)$ is unimodular if $\pairing_\rG$ induces an isomorphism $\rG \to \Hom (\rG,\Z)$.
\end{definition}

\begin{definition}
    An automorphism $\rS_\rG\colon \rG \to \rG$ of a pseudolattice $(\rG,\pairing_\rG)$ is a Serre operator if it satisfies
    \[\pair{v_1,v_2}_\rG = \pair{v_2,\rS_\rG(v_1)}_\rG \qquad \mbox{ for all } v_1,v_2\in \rG\]
\end{definition}

If a Serre operator exists then it is unique by the non-degeneracy of $\pairing_\rG$. Moreover, it has a convenient matrix description: if $\chi$ is a Gram matrix of $\pairing_\rG$, then, provided that $\chi^{-1}$ is an integral matrix, $\chi^{-1} \chi^T$ is a Serre operator. In particular, if $G$ is unimodular then $\chi^{-1}$ is integral, hence all unimodular pseudolattices have a Serre operator. 

An important class of unimodular pseudolattices is given by those admitting  an exceptional basis:

\begin{definition}
    An element $e\in \rG$ is called exceptional if $\pair{e,e}_\rG=1$. A sequence of elements $e_\bullet=(e_1,\ldots,e_n)$ is exceptional if every $e_i$ is exceptional, and $\pair{e_i,e_j}_\rG=0$ for all $i>j$. An exceptional basis of $\rG$ is an exceptional sequence whose elements form a basis of $\rG$.
\end{definition}

If $\rG$ admits an exceptional basis, then $\rG$ is unimodular as the Gram matrix of the pairing written in the exceptional basis is triangular with $1$ on the diagonal.

Let $e\in \rG$ be an exceptional element. 

\begin{definition}
\label{def:num-mutation}
    The left and right mutation with respect to $e$ are endomorphisms of $\rG$ defined, respectively, by 
    \begin{equation}
        \sfL_e(v) = v - \pair{e,v}e, \qquad \sfR_e(v) = v - \pair{v,e}e
    \end{equation}
\end{definition}

Given an exceptional sequence $e_\bullet = (e_1,\ldots, e_n)$, the sequences
\begin{equation}
\begin{alignedat}{3}
        \sfL_{i}(e_\bullet) & \coloneqq (e_1,\ldots, e_{i-1}, \sfL_{e_{i}}(e_{i+1}), e_i,  e_{i+2},\ldots, e_n)\\
        \sfR_{i}(e_\bullet) & \coloneqq (e_1,\ldots, e_{i-1},  e_{i+1},  \sfR_{e_{i+1}}(e_i),  e_{i+2},\ldots, e_n)
\end{alignedat}
\end{equation}
(here $i=0,\ldots,n-1$) are both exceptional, and these two operations are mutual inverses.

\begin{definition}
A pseudolattice $\rG$ is \textit{surface-like} if there exists a primitive element $\bfp\in \rG$, called a \textit{point-like} object, satisfying:
\begin{itemize}
    \item $\pair{\bfp,\bfp}_\rG=0$;
    \item $\pair{\bfp,v}_\rG = \pair{v,\bfp}_\rG$ for all $v\in \rG$;
    \item The bilinear form $\pairing_\rG$ is symmetric on $\bfp^\perp = {}^\perp\!\bfp$. 
\end{itemize}
\end{definition}

One then defines the \textit{rank} of an element $v$ in a surface-like pseudolattice $\rG$ to be 
\begin{equation}
    \label{eq:defRank}
    \rk(v) \coloneqq \pair{\bfp,v}_\rG = \pair{v,\bfp}_\rG
\end{equation}

There is a filtration 
\begin{equation}
    \label{eq:filtrationPseudolattice}
  \Z\bfp\subseteq \bfp^\perp=\ker(\rk) \subseteq \rG. 
\end{equation}
The middle factor $\bfp^\perp/\bfp$ is called the N\'eron--Severi group of $\rG$ and is denoted by $\NS(\rG)$. The pairing $\chi$ restricts on $\NS(\rG)$ to a \textit{symmetric} bilinear form.
When it exists, the Serre operator acts as the identity on the factors of \eqref{eq:filtrationPseudolattice}. 

In \cite[Section 4]{Kuz17_pseudolattices}, the problem of classifying surface-like pseudolattices leads to considering the \textit{norm} of an exceptional basis $e_\bullet$, defined as the quantity 
\[ \left\lVert e_\bullet \right\rVert \coloneqq \sum_{i=1}^n \rk(e_i)^2. \]
We say that $e_\bullet$ is norm-minimal if the norm of any exceptional basis obtained from $e_\bullet$ by a sequence of mutations is greater than or equal to the norm of $e_\bullet$.

\begin{example}[Del Pezzo surfaces]\label{ex:DP}
Let $X_d$ be a del Pezzo surface of degree $d$ that is a blowup of $\pr 2$ in $\ell=9-d$ distinct points.
By Orlov's blowup formula, there is a semiorthogonal decomposition (SOD)
\begin{equation}
    D^b(\Coh(X_d))=\left\langle D^b(\pr 2), \sO_{E_1},\ldots,\sO_{E_{\ell}}  \right\rangle
\end{equation}
where the $E_i$ are the exceptional divisors of the blow-up. 

We abbreviate $K_0(D^b(\Coh(X_d)))$ by $K_0(X_d)$.  The SOD above implies that the Euler pairing $\chi$ on $K_0(X_d)$ is non-degenerate, hence $K_0(X_d)$ coincides with the numerical Grothendieck group. Then, equipped with the Euler pairing, $K_0(X_d)$ is a pseudolattice of rank $3+\ell$.  A Serre operator is defined by the action of the Serre functor $- \otimes \omega_{X_d}[2]$ of $D^b(\Coh(X_d))$ on $K_0(X_d)$.

By picking, for example, the exceptional collection $D^b(\Coh(\pr 2))=(\sO_{\pr 2}, \sT(-1), \sO_{\pr 2}(1))$ we obtain an exceptional basis for $K_0(X_d)$: 
\begin{equation}
    ([\sO_{X_d}], [\pi^*\sT(-1)], [\sO_{X_d}(1)], [\sO_{E_1}],\ldots,[\sO_{E_\ell}])
\label{eq:ExcBasisDP}
\end{equation}
Here, $[-]$ denotes the class of an object of $D^b(\Coh(X_d))$ in $K_0(X_d)$.
The Gram matrix of the Euler pairing in the basis \eqref{eq:ExcBasisDP} reads
\begin{equation}\label{eq:GramDP}
{%\scriptsize
M_\ell\coloneqq \left(\begin{array}{c|c}
    \begin{array}{ccc}
      1  & 3 & 3        \\
       0  & 1 & 3      \\
        0  & 0 & 1      
    \end{array} &
    \begin{array}{ccc}
         1 & \ldots &  1 \\
        2 & \ldots &  2 \\
          1 & \ldots &  1 
    \end{array} \\
    \hline
    &\\
    \text{\large $0_{\ell \times 3}$} & \text{\large $I_{\ell\times \ell}$}\\
    &\\
    \end{array}\right)
    }
\end{equation}

The pseudolattice $(K_0(X_d),\chi)$ is surface-like with point-like element $\bfp\coloneqq [\sO_p]$ the class of a skyscraper sheaf on $X_d$. With this choice, the N\'eron--Severi lattice of $K_0(X_d)$ is identified with $\NS(X_d)$. A basis for $\NS(X_d)$ is given by the cohomology classes of the strict transform of a line from $\pr 2$ and of each exceptional divisor. Then,  it is immediate to see that $\NS(X_d)$ is the rank $\ell+1$ lattice $\sfI^{1,\ell}$, defined as $\Z^{\ell+1}$ equipped with the bilinear form $\mathrm{diag}(1,-1,\ldots,-1)$ with respect to the standard basis. 
Under this identification, a canonical divisor on $X_d$ has class 
\[ \bfk_\ell = (-3,1,\ldots,1). \]
The vector $\bfk_\ell$ 
is a canonical class in the sense of \cite[Section 3.3]{Kuz17_pseudolattices}. 
\end{example}

\begin{example}[Genus-$1$ Lefschetz fibrations]
\label{exa:HT}
    In Section \ref{sec:construct-Lf} we construct genus-$1$ Lefschetz fibrations $\wt_d \colon \Yt_d \to \CC$ with $d=1,2,3$ and equip their relative homology groups $H_2(\Yt_d,\Sigma;\Z)$ with the Seifert pairing. The result is a surface-like pseudolattice
\begin{equation} 
\label{eq:defOfHl}
\left( H_2(\Yt_d,\Sigma; \Z), \pairing_{\mathrm{Sft}} \right)    
\end{equation}
More generally, pseudolattices associated with genus-$1$ Lefschetz fibrations are studied in \cite[Section 4.2]{HarderThompson20}, where they are also constructed abstractly by gluing smaller pseudolattices along a spherical homomorphism. % \cite[Section 2.4]{HarderThompson20}.
The homomorphism in question is called the \emph{asymptotic charge} 
(see Section \ref{sec:GHS}). 
The Serre operator is induced by the action of global monodromy on the fibration. % \cite[Proposition~4.4]{HarderThompson20}.
By \cite[Corollary 4.17]{HarderThompson20} and \cite[Theorem 5.1]{HarderThompson20}, $H_2(\Yt_d,\Sigma; \Z)$ is isometric to $K_0(X_d)$. 

See \cite[Section 4]{GiovenzanaThompson} for a study of these pseudolattices for (non-Lefschetz) elliptic fibrations.
\end{example}

\begin{remark}\label{rmk:AKO-collections}
Let $(M_\ell, W_\ell)$ be the LG models constructed by \cite{AKO06} (see Section \ref{sec:geometry-cats}).
Choosing $W_\ell$ as a small perturbation of $W_0$ (so that three critical values lie near those of $W_0$ and the reminaing $\ell$ lie near $\infty$), 
\cite{AKO06} produces a full exceptional collection for $D^b(\lagvc(W_\ell))$. 
The exceptional collection gives rise to an exceptional basis for the pseudolattice $K_0(W_\ell) \simeq H_2(M_\ell,\Sigma;\Z)$ (where $\Sigma$ is a regular fiber). The Gram matrix of the pairing written in this exceptional basis coincides precisely with \eqref{eq:GramDP} (for $d=9-\ell$),
without the need to apply mutations. In particular, the Grothendieck groups $K_0(X_d)$ and $K_0(W_\ell)$ are immediately seen to be isometric.
\end{remark}

\subsection{Sequences of mutations}
\label{sec:thm}
Let $(Y^\prime, w^\prime)$ be an LG model obtained by restricting  to $\CC$ 
a rational minimal elliptic surface $S$ with a fibre $D$ of type $I_{d}$, with $d=1,2,3$.
This includes the LG models $(Y_d,w_d)$ in Equation \eqref{eq:Yw},   $(\widetilde{Y}_d,\widetilde{w}_d)$ in Equation \eqref{eq:TildeYw}, and  $(M_\ell,W_\ell)$ in \cite{AKO06}. 
Its Fukaya-Seidel category $\mathrm{FS}(Y^\prime, w^\prime)$ is equivalent to $D^b\lagvc(\wt_d)$, arguing as in the proof of Proposition \ref{prop:EquivalenceYtildeAko}, or as in Remark \ref{rem:Y-mirror}.
In particular, we may identify the corresponding Grothendieck groups $K_0(w^\prime)$, which, by Example \ref{exa:HT}, or, alternatively, Remark \ref{rmk:AKO-collections},  are all isometric to $K_0(X_d)$. Moreover, all possible exceptional bases produced by such LG models are related to one-another by a sequence of mutations.

In this Section we exhibit explicit sequences of mutations  which send the bases \eqref{eq:vc-dP1},\eqref{eq:vc-dP2}, \eqref{eq:vc-dP3} of the pseudolattices $H_2(\Yt_d,\Sigma; \Z)=K_0(\wt_d)$ in \eqref{eq:defOfHl} to  bases 
with respect to which
the Gram matrix of the pairing 
coincides with the matrix $M_\ell$ in \eqref{eq:GramDP}. 

Thus, Theorem \ref{thm:MutationSequence} relates in a direct way the exceptional bases arising from our construction of Section \ref{sec:ConstructionLefschetzFibrations} to those appearing in \cite{AKO06}.

\begin{thm}\label{thm:MutationSequence}
The following sequences of mutations map the exceptional bases \eqref{eq:vc-dP3}, \eqref{eq:vc-dP2}, and \eqref{eq:vc-dP1} to the exceptional bases
\begin{equation}\label{eq:BasisAfterMutations}
a+b,2a-b,a-2b,\underbrace{-b,\ldots,-b}_{\ell \mbox{ copies}}   
\end{equation}
whose associated Gram matrices coincide with the matrices $M_\ell$ in \eqref{eq:GramDP}, $\ell=6,7,8$.

\begin{align}
    \label{eq:mutationsFord3}
    d=3 \colon \quad & \mathbf{\beta}_3 \coloneqq \sfL_1 \cdot (\sfL_2\sfL_3) \cdot \sfL_1 \cdot \sfL_3 \cdot \sfL_1 \cdot (\sfL_4 \sfL_5 \sfL_6  \sfL_7) 
 \cdot (\sfL_3 \sfL_4 \sfL_5) \cdot (\sfL_2  \sfL_3) \cdot \sfL_1\\   \label{eq:mutationsFord2}    d=2 \colon \quad & \beta_2\coloneqq  \beta_3  \cdot(\sfR_8 \cdot \ldots \cdot \sfR_1) \cdot (\sfR_8\cdot \sfR_7 \cdot \sfL_4)\\
 \label{eq:mutationsFord1} d=1 \colon \quad & \beta_1 \coloneqq \beta_2 \cdot (\sfR_9 \cdot \ldots \cdot \sfR_4) \cdot \sfL_6 
    \end{align}
\end{thm}

\begin{remark}
\label{rem:change-basis}
 The change of basis $a \mapsto -a$, $b\mapsto -(a+b)$ turns \eqref{eq:BasisAfterMutations}  into the exceptional basis obtained in \cite[Section 3.3]{AKO06}.    
\end{remark}

Let $(Y^\prime, w^\prime)$ be an LG model as above.
If $(Y^\prime, w^\prime)$ defines a Lefschetz fibration, Theorem \ref{thm:MutationSequence} provides a tool to obtain an explicit sequence of mutations realizing an isometry between $K_0(X_d)$ and the pseudolattice  $H_2(Y^\prime,\Sigma;\Z)$ (with $\Sigma$ a general fiber), mapping the exceptional basis of vanishing thimbles of $(Y',w')$ to the basis \eqref{eq:ExcBasisDP}. 
Indeed, in this case $Y'$ admits a Weierstrass equation: interpolating it with the equation of $\Yt$, 
one can keep track of the trajectories of the critical values, and these in turn yield a sequence of mutations relating the corresponding pseudolattices. %Fukaya--Seidel Categories. 
For example, if two critical values $\lambda_1,\lambda_2$ with associated exceptional pair $(L_1,L_2)$ with respect to arcs $(\gamma_1, \gamma_2)$ move as in Figure \ref{fig:Mutations_VC}, then the exceptional pair $(L_1',L_2')$ associated to the arcs $(\gamma_1',\gamma_2')$ is equivalent to the right mutation $(L_1,L_2) \mapsto 
%(L_1',L_2')=
(L_2,\sfR_{L_2}(L_1))$.

\begin{figure}[ht!]
\centering
\begin{tikzpicture}[scale=0.8, font=\footnotesize]
\node (*) at (0,0) [dot, label=below:{$\lambda_*$}]{};
\node (A) at (-1,2) [dot, label=above:{$\lambda_1$}]{};
\node (B) at (1,2) [dot, label=above:{$\lambda_{2}$}]{};
\node (A') at (2,3) [dot]{}; 
\node (B') at (-2,1) [dot]{};

\draw (A) edge[dashed,->,out=80, in=130, looseness=1] (A');
\draw (B) edge[dashed,->,out=-40, in=-30, looseness=1] (B');

\draw (*) edge[swap,"{$\gamma_1$}"] (A);
\draw (*) edge[swap, "{$\gamma_{2}$}"] (B);

\node (*) at (7,0) [dot, label=below:{$\lambda_*$}]{};
\node (C') at (9,3) [dot]{}; 
\node (D') at (5,1) [dot]{};

\draw (*) edge[swap, "{$\gamma_{2}'$}"] (C');
\draw (*) edge["{$\gamma_{1}'\sim\gamma_{2}$}"] (D');
\end{tikzpicture}
\caption{On the left: the critical values $\lambda_1,\lambda_2$, the arcs $\gamma_1,\gamma_2$, and the trajectories of $\lambda_1,\lambda_2$ via the interpolation. On the right: the arcs $\gamma_1^\prime, \gamma_2^\prime$.}
\label{fig:Mutations_VC}
\end{figure}
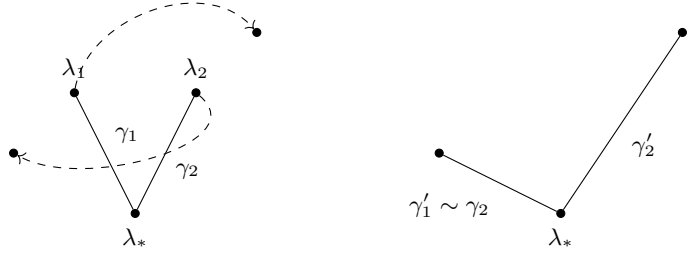

\subsection{Proof of Theorem \ref{thm:MutationSequence}}
\label{sec:thmproof}

We divide the proof in two steps. First, we provide a sequence of mutations for the degree $3$ case. Then, we deduce associated sequences of mutations for degrees $2$ and $1$ by deforming the Weierstrass fibrations of the elliptic surfaces $\widetilde{F}_d \colon \widetilde{S}_d \to \pr 1$ of Section \ref{sec:construct-Lf}, $d=3,2,1$, into one another.
We achieve this by
interpolating the polynomials $\widetilde{p}_d$  in Equation \eqref{eq:dP-def}.
The second step of the proof is an algebraic version of the construction of \cite{AKO06} outlined in Section \ref{sec:geometry-cats}.

Before the proof, we explain how we found the sequence of mutations for $d=3$.
The exceptional basis \eqref{eq:ExcBasisDP} is nearly norm minimal. We can obtain a norm minimal one by considering an exceptional collection of line bundles on $\pr 2$ instead of $(\sO_{\pr 2},\sT(-1),\sO_{\pr2}(1))$. 
Thus, one might expect that mutations of the basis \eqref{eq:vc-dP3} which reduce its norm will bring us closer to the basis \eqref{eq:ExcBasisDP}. This is indeed the case, and produces (among possibly other sequences of mutations) the sequence \eqref{eq:mutationsFord3}. 

To compute norms, we need a point-like object. Since the Serre operator $\rS \coloneqq \rS_{K_0(\wt_d)}$ act as the identity on the factors of \eqref{eq:filtrationPseudolattice}, we have $\im(\id - \rS)^2 = \Z\bfp$. Then, $\bfp$ may be chosen as a primitive vector in $\im(\id - \rS)^2$, which is unique up to a sign. This is completely explicit, since $\rS$ may be expressed in the basis \eqref{eq:vc-dP3} as the product $M^{-1}M^t$, where $M$ is the Gram matrix \eqref{eq:gram-3}.

\subsubsection{Proof for $d=3$.}

We begin by recalling the sequence \eqref{eq:vc-dP3} of vanishing cycles:
\[a+b,b,a,b,a,b,a,b,a.\]
To this we apply the sequence of mutations $\beta_3$:

\begin{equation}\label{eq:mutationSequence_dp3}
\begin{split}
 \sfL_1\colon \quad & a+b,{\color{BurntOrange}a-b},b,b,a,b,a,b,a  \\
 \sfL_2 \circ \sfL_3\colon \quad  & a+b,a-b,{\color{BurntOrange}a-2b},b,b,b,a,b,a \\
 \sfL_3 \circ \sfL_4 \circ \sfL_5\colon \quad   & a+b,a-b,a-2b,{\color{BurntOrange}a-3b},b,b,b,b,a \\
 \sfL_4 \circ \sfL_5 \circ \sfL_6 \circ \sfL_7\colon \quad  & a+b,a-b,a-2b,a-3b,{\color{BurntOrange}a-4b},b,b,b,b  \\
 \sfL_3 \circ \sfL_1\colon \quad & a+b,{\color{BurntOrange}-b},a-b,{\color{BurntOrange}-b},a-3b,b,b,b,b \\
 \sfL_1\colon \quad  & a+b,{\color{BurntOrange}a-2b},-b,-b,a-3b,b,b,b,b \\
 \sfL_2 \circ\sfL_3\colon \quad  & a+b,a-2b,{\color{BurntOrange}a-5b},-b,-b,b,b,b,b \\
 \sfL_1\colon \quad  & a+b,{\color{BurntOrange}2a-b},a-2b,-b,-b,-b,-b,-b,-b 
\end{split}
\end{equation}
In each row, the new classes resulting from the mutations at the beginning of the row are colored in orange. In the last row some orientations (at positions $1,5,6,7,8$) were switched: this is to match exactly the Gram matrix $M_\ell$ of \eqref{eq:GramDP}.

\subsubsection{Proof for $d=2$ and $d=1$.}

First, consider the elliptic surface $\widetilde{F}_3 \colon \widetilde{S}_3 \to \pr 1$.
To it we can associate the collection \eqref{eq:vc-dP3} 
extended by three identical cohomology classes,  which we indicate by $c_3$: 
 \begin{equation}
 \label{eq:index3-extended}
     a+b,b,a,b,a,b,a,b,a, c_3, c_3,c_3  
 \end{equation} 
 Indeed, $\widetilde{F}_3$ has a fiber of type $I_3$ over $\infty$, contributing three additional vanishing cycles with the same cohomology class.

By the previous step, the sequence of mutations $\beta_3$ takes \eqref{eq:index3-extended} to 
\begin{equation} 
\label{eq:final-12-els}
a+b, 2a-b, a-2b, -b, -b,-b,-b,-b,-b,  c_3,c_3,c_3 
\end{equation}

A simple monodromy computation as in \cite[Lemma~3.1]{AKO06} (applied to either \eqref{eq:index3-extended} or \eqref{eq:final-12-els}) shows that $c_3=\pm b$
\footnote{This is consistent with the fact that the collection \eqref{eq:final-12-els} has to match the cycles $L_i$ in  \cite[Lemma 3.1]{AKO06} after the change of basis in Remark \ref{rem:change-basis}.}. 
Choosing $c_3=-b$, for each $\ell$ the the collection \eqref{eq:BasisAfterMutations} is the subcollection of \eqref{eq:final-12-els} given by the first $3+\ell$ elements.

We now consider the collections \eqref{eq:vc-dP2} and  \eqref{eq:vc-dP1}.

\subsubsection*{Interpolation between $d=3$ to $d=2$.}
\label{ssec:interpol32}
With the same procedure as above, to the elliptic surface $\widetilde{F}_2 \colon \widetilde{S}_2 \to \pr 1$
we can associate the collection \eqref{eq:vc-dP2} 
extended by two identical cohomology classes, which we denote by $c_2$:
\begin{equation}
 \label{eq:index2-extended}
     a+b,{\boldsymbol{a}},b,a,a,a-b,b,b,a,b,  c_2,c_2  
 \end{equation} 
 (The class shown in bold will be relevant below). As above, a simple monodromy computation applied to \eqref{eq:index2-extended} shows that $c_2=\pm b$.

We define a family of Weierstrass fibrations
by interpolating the polynomials $\widetilde{p}_3$ and $\widetilde{p}_2$ as follows
\begin{equation}
\label{eq:us}
u_s\coloneqq e^{i \pi s}\widetilde{p}_3 + s(\widetilde{p}_3 + \widetilde{p}_2), \qquad s \in [0,1]  \end{equation}
Observe that $u_0=\widetilde{p}_3$ and $u_1=\widetilde{p}_2$, that is, the Weierstrass fibrations at $s=0$ and $s=1$ are those corresponding to  $\widetilde{F}_3 \colon \widetilde{S}_3 \to \pr 1$ and $\widetilde{F}_2 \colon \widetilde{S}_2 \to \pr 1$ respectively. This specific choice is more convenient than, for example, the interpolation $s\widetilde{p}_2 + (1-s)\widetilde{p}_3$ when we plot the trajectories of critical values.

Formula \eqref{eq:us} and Figure \ref{fig:Interpol_dp2_dp3},  which 
tracks the trajectories of critical values over $\CC$ as $s$ varies, 
suggest a sequence of mutations relating the two collections \eqref{eq:index2-extended} and \eqref{eq:index3-extended}. 

First, observe that as  $s$ increases the finite critical values of $\widetilde{F}_3$ 
are perturbed within $\CC$. 
Among the three critical points over $\infty$, two remain at $\infty$, which indicates that $c_2=c_3$.
The remaining one is displaced to a finite value of the potential (for $s=1$, this value is in position $1$ and the corresponding class is marked in bold in \eqref{eq:index3-extended}).  
Altogether, this indicates that a mutation sequence
will involve only the first $10$ classes of \eqref{eq:index2-extended}.

The trajectories in Figure \ref{fig:Interpol_dp2_dp3} 
point to one of the mutation sequences:
\begin{equation}
        \label{eq:finite-mutations-d=2}
(\sfR_8 \cdot \ldots \cdot \sfR_1) \cdot (\sfR_8\cdot \sfR_7 \cdot \sfL_4)= (\sfR_7\cdot \sfR_6 \cdot \sfL_3) \cdot (\sfR_8 \cdot \ldots \cdot \sfR_1)
\end{equation}

Applying $(\sfR_8\cdot \sfR_7 \cdot \sfL_4)$ to \eqref{eq:index2-extended}, we obtain (up to orientations)
the collection
\[a+b,{\boldsymbol{a}},b,a,{\color{BurntOrange}b},a,b,a,b,{\color{BurntOrange}a}  ,  b,b\]
Then $(\sfR_8 \cdot \ldots \cdot \sfR_1)$ moves the class in bold to the desired position 9. The resulting collection is
\[  a+b,b,a,b,a,b,a,b,a,{\boldsymbol{b}},  b,b \]
i.e. (up to orientation of the last three classes) the collection \eqref{eq:index3-extended}. 
 
It follows that the sequence $\beta_2$ in \eqref{eq:mutationsFord2}, obtained by composing $\beta_3$ and \eqref{eq:finite-mutations-d=2},  sends \eqref{eq:vc-dP2} to \eqref{eq:BasisAfterMutations} with $\ell=7$.

\subsubsection*{Interpolation between $d=2$ and $d=1$}\label{ssec:interpol21} 

As above, we start from \eqref{eq:vc-dP1} ``extended over infinity'' by a class $c_1$
(and as above, $c_1=\pm b$)
\begin{equation}
 \label{eq:index1-extended}  
 a+b,a,b,a,b,a,b,a,b,a,b, c_1
\end{equation}
and define a family of Weierstrass fibrations by the interpolation:
\[ v_s\coloneqq e^{i \pi s}\widetilde{p}_2 + s(\widetilde{p}_2 + \widetilde{p}_1), \qquad s \in [0,1],
\]
Now Figure \ref{fig:Interpol_dp1_dp2} shows that a cycle coming from infinity in \eqref{eq:index2-extended} is inserted in position $4$, and suggests the sequence of mutations:
\[(\sfR_9 \cdot \ldots \cdot \sfR_4) \cdot \sfL_6 = \sfL_5 \cdot (\sfR_9 \cdot \ldots \cdot \sfR_4)
\]
which turns \eqref{eq:index1-extended} into \eqref{eq:index2-extended}. 
One concludes composing this sequence with $\beta_2$. \qed

\tikzset{every picture/.style={line width=0.75pt}} %set default line width to 0.75pt        

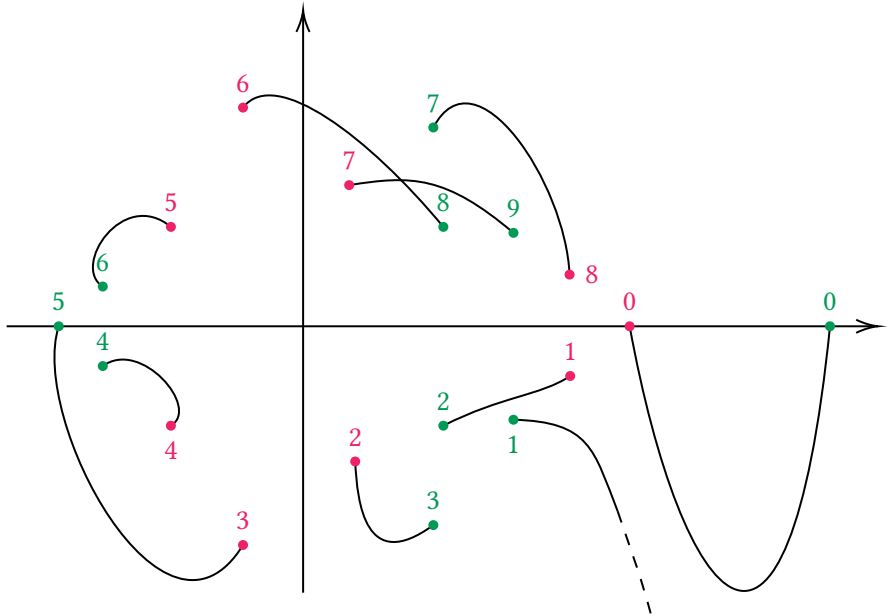
\begin{figure}[ht!]
\begin{tikzpicture}[x=0.75pt,y=0.75pt,yscale=-1,xscale=1]
%uncomment if require: \path (0,413); %set diagram left start at 0, and has height of 413

%Straight Lines [id:da5307764970982063] 
\draw [line width=0.75]    (152,170) -- (585,170) ;
\draw [shift={(587,170)}, rotate = 180] [color={rgb, 255:red, 0; green, 0; blue, 0 }  ][line width=0.75]    (10.93,-3.29) .. controls (6.95,-1.4) and (3.31,-0.3) .. (0,0) .. controls (3.31,0.3) and (6.95,1.4) .. (10.93,3.29)   ;
%Straight Lines [id:da9271855135295264] 
\draw [line width=0.75]    (300,304) -- (300,13) ;
\draw [shift={(300,11)}, rotate = 90.39] [color={rgb, 255:red, 0; green, 0; blue, 0 }  ][line width=0.75]    (10.93,-3.29) .. controls (6.95,-1.4) and (3.31,-0.3) .. (0,0) .. controls (3.31,0.3) and (6.95,1.4) .. (10.93,3.29)   ;
%Curve Lines [id:da16602139372123847] 
%OLD COLOR
%\draw [color={rgb, 255:red, 245; green, 166; blue, 35 }  ,draw opacity=1 ]   (463,170) .. controls (489,313) and (542,379) .. (563,170) ;
%\draw (463,170)  node [wdot]{};
%\draw (563,170) node [dot]{};
\draw (463,170) .. controls (489,313) and (542,379) .. (563,170) ;
\draw (463,170)  node [pdot, label={\textcolor{WildStrawberry}{0}}]{};
\draw (563,170) node [gdot, label={\textcolor{ForestGreen}{0}}]{};
%Curve Lines [id:da16403300850656177] 
\draw (370,220) .. controls (399.33,205) and (418.33,205) .. (433.33,195) ;
\draw (370,220)  node [gdot, label={\textcolor{ForestGreen}{2}}]{};
\draw (433.33,195) node [pdot, label={\textcolor{WildStrawberry}{1}}]{};
%Curve Lines [id:da9570202124877363] 
\draw (326,238) .. controls (327,259) and (331,295) .. (365,270) ;
\draw (326,238)  node [pdot, label={\textcolor{WildStrawberry}{2}}]{};
\draw (365,270) node [gdot, label={\textcolor{ForestGreen}{3}}]{};
%Curve Lines [id:da26880192951874426] 
\draw (178,170) .. controls (163,217) and (231,344) .. (270,280) ;
\draw (178,170)  node [gdot, label={\textcolor{ForestGreen}{5}}]{};
\draw (270,280) node [pdot, label={\textcolor{WildStrawberry}{3}}]{};
%Curve Lines [id:da8380442143629385] 
\draw (200,190) .. controls (219,176) and (249,211) .. (234,220) ;
\draw (200,190)  node [gdot, label={\textcolor{ForestGreen}{4}}]{};
\draw (234,220) node [pdot, label=below:{\textcolor{WildStrawberry}{4}}]{};
%Curve Lines [id:da20148272896439745] 
\draw (200,150) .. controls (184,140) and (210,100) .. (234,120) ;
\draw (200,150)  node [gdot, label={\textcolor{ForestGreen}{6}}]{};
\draw (234,120) node [pdot, label={\textcolor{WildStrawberry}{5}}]{};
%Curve Lines [id:da4601286132033249] 
\draw (270,60) .. controls (286,39) and (333,75) .. (370,120) ;
\draw (270,60)  node [pdot, label={\textcolor{WildStrawberry}{6}}]{};
\draw (370,120) node [gdot, label={\textcolor{ForestGreen}{8}}]{};
%Curve Lines [id:da39812804931211376] 
\draw (323,99) .. controls (354,94) and (371,94) .. (405,123) ;
\draw (323,99)  node [pdot, label={\textcolor{WildStrawberry}{7}}]{};
\draw (405,123) node [gdot, label={\textcolor{ForestGreen}{9}}]{};
%Curve Lines [id:da36660150344111697] 
\draw (365,70) .. controls (386,30) and (431,98) .. (433,144) ;
\draw (365,70)  node [gdot, label={\textcolor{ForestGreen}{7}}]{};
\draw (433,144) node [pdot, label=right:{\textcolor{WildStrawberry}{8}}]{};
%Curve Lines [id:da21126299291520745] 
\draw (405,217) .. controls (443.33,218) and (445,232) .. (457,263) ;
\draw (405,217)  node [gdot, label=below:{\textcolor{ForestGreen}{1}}]{};
%Curve Lines [id:da1556523013908062] 
\draw [dash pattern={on 4.5pt off 4.5pt}]  (457,263) .. controls (464.84,286.52) and (466.92,289.87) .. (474.53,318.23) ;

\end{tikzpicture}
\vspace{-0.3 cm}

\caption{Critical values of $u_s$. Bullets \textcolor{ForestGreen}{$\bullet$} with green labels
mark critical values of $u_1=\widetilde{p}_2$ and pink ones \textcolor{WildStrawberry}{$\bullet$}
mark critical values of $u_0=\widetilde{p}_3$.
}
\label{fig:Interpol_dp2_dp3}
\end{figure}

\begin{figure}[ht!]
    \centering
\tikzset{every picture/.style={line width=0.75pt}} %set default line width to 0.75pt        

\begin{tikzpicture}[x=0.75pt,y=0.75pt,yscale=-1,xscale=1]
%uncomment if require: \path (0,413); %set diagram left start at 0, and has height of 413

%Straight Lines [id:da5307764970982063] 
\draw [line width=0.75]    (152,170) -- (585,170) ;
\draw [shift={(587,170)}, rotate = 180] [color={rgb, 255:red, 0; green, 0; blue, 0 }  ][line width=0.75]    (10.93,-3.29) .. controls (6.95,-1.4) and (3.31,-0.3) .. (0,0) .. controls (3.31,0.3) and (6.95,1.4) .. (10.93,3.29)   ;
%Straight Lines [id:da9271855135295264] 
\draw [line width=0.75]    (300,304) -- (300,13) ;
\draw [shift={(300,11)}, rotate = 90.39] [color={rgb, 255:red, 0; green, 0; blue, 0 }  ][line width=0.75]    (10.93,-3.29) .. controls (6.95,-1.4) and (3.31,-0.3) .. (0,0) .. controls (3.31,0.3) and (6.95,1.4) .. (10.93,3.29)   ;
%Curve Lines [id:da16602139372123847] %arrov 0 0
\draw (563,170) .. controls (589,312) and (482,380) .. (503,170) ;
\draw (563,170) node [bdot, label={\textcolor{Plum}{0}}]{};
\draw (503,170) node [gdot, label={\textcolor{ForestGreen}{0}}]{};
%Curve Lines [id:da16403300850656177] % arrow 2 2
\draw (365,240) .. controls (375,231) and (359,221) .. (350,230) ;
\draw (365,240) node [gdot, label=below:{\textcolor{ForestGreen}{2}}]{};
\draw (350,230) node [bdot, label={\textcolor{Plum}{2}}]{};
%Curve Lines [id:da9570202124877363] % arrow3 3
\draw (280,240) .. controls (333,244) and (311,265) .. (340,280) ;
\draw (280,240) node [bdot, label={\textcolor{Plum}{3}}]{};
\draw (340,280) node [gdot, label={\textcolor{ForestGreen}{3}}]{};
%Curve Lines [id:da26880192951874426] % arrow 5 7
\draw (180,170) .. controls (164,131) and (185,84) .. (250,80) ;
\draw (180,170) node [gdot, label={\textcolor{ForestGreen}{5}}]{};
\draw (250,80) node [bdot, label={\textcolor{Plum}{7}}]{};
%Curve Lines [id:da8380442143629385] %arrow 4 5
\draw (200,190) .. controls (186,207) and (208,234) .. (230,210) ;
\draw (200,190) node [gdot, label={\textcolor{ForestGreen}{4}}]{};
\draw (230,210) node [bdot, label={\textcolor{Plum}{5}}]{};
%Curve Lines [id:da20148272896439745] %arrow 6 6
\draw (200,150) .. controls (238,148) and (233,120) .. (230,130) ;
\draw (200,150) node [gdot, label={\textcolor{ForestGreen}{6}}]{};
\draw (230,130) node [bdot, label={\textcolor{Plum}{6}}]{};
%Curve Lines [id:da4601286132033249] %arrow 9 8
\draw (350,120) .. controls (369,110) and (372,104) .. (365,100) ;
\draw (350,120) node [bdot, label={\textcolor{Plum}{9}}]{};
\draw (365,100) node [gdot, label={\textcolor{ForestGreen}{8}}]{};
%Curve Lines [id:da39812804931211376] %arrow 9 10
\draw (450,130) .. controls (460,106) and (413,77) .. (400,110) ;
\draw (450,130) node [bdot, label=below:{\textcolor{Plum}{10}}]{};
\draw (400,110) node [gdot, label={\textcolor{ForestGreen}{9}}]{};
%Curve Lines [id:da36660150344111697] %arrow 8 7
\draw (340,60) .. controls (319,29) and (267,29) .. (280,70) ;
\draw (340,60) node [gdot, label={\textcolor{ForestGreen}{7}}]{};
\draw (280,70) node [bdot, label=below:{\textcolor{Plum}{8}}]{};
%Curve Lines [id:da21126299291520745] %arrow 1 1
\draw (400,230) .. controls (442,244) and (495,201) .. (450,210) ;
\draw (400,230) node [gdot, label={\textcolor{ForestGreen}{1}}]{};
\draw (450,210) node [bdot, label={\textcolor{Plum}{1}}]{};
%Curve Lines [id:da3449022842847702] 
\draw (250,260) .. controls (271,319) and (303,278) .. (321,344) ;
\draw (250,260) node [bdot, label={\textcolor{Plum}{4}}]{};
%Curve Lines [id:da8891183466394776] 
\draw [dash pattern={on 4.5pt off 4.5pt}]  (321,344) .. controls (332.88,373.7) and (317.32,341.65) .. (344.17,384.67) ;

\end{tikzpicture}
\caption{Critical values of $v_s$. Green bullets \textcolor{ForestGreen}{$\bullet$}
mark critical values of $v_0=\widetilde{p}_2$, 
violet \textcolor{Plum}{$\bullet$} and labels mark critical values of $v_1=\widetilde{p}_1$.}
\label{fig:Interpol_dp1_dp2}
\end{figure}
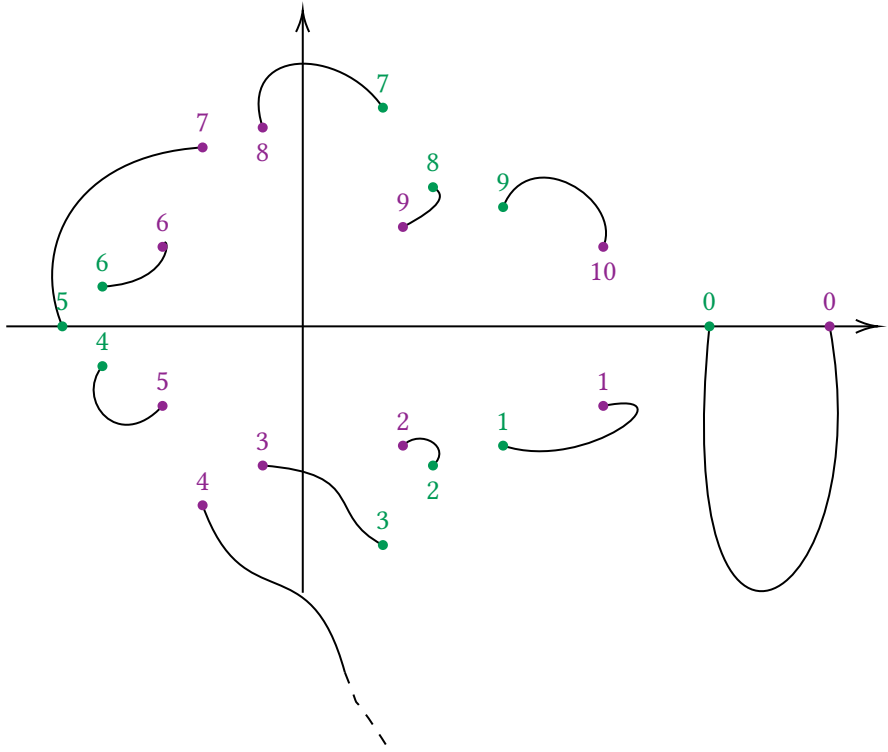

\subsection{Comparison with string junction results}
\label{sec:GHS}
In \cite{GHS13:matter}, the authors view a local equation of an ADE surface singularity as a  Weierstrass form
\begin{equation}
\label{eq:GHSSing}
y^2=x^3+a(z)x+b(z) \end{equation}
defining an affine singular surface $Z$
with a genus-$1$ fibration $Z \to D$ over a disc. Here, $z$ is a coordinate on $D$. Up to shrinking $D$, the discriminant is $\Delta=z^N$, where $N=\ell+1$ in the $A_\ell$ case, and $N=\ell+2$ in the $D_\ell$ and $E_\ell$ cases. 
Then, they deform \eqref{eq:GHSSing} to a Lesfchetz genus-$1$ fibration $\widetilde{q} \colon \widetilde{Z} \to D$ with $N$ distinct critical values.

We focus on the $E_\ell$ case, and denote the corresponding Lefschetz fibration by $\widetilde{q}_\ell$. Our fibration  $\widetilde{w}_d \colon \widetilde{Y}_d \to \CC$, where $d=9-\ell$, is a global version of the fibration $\widetilde{q}_\ell$.

\subsubsection{Collections of vanishing cycles}\label{ssec:GHSCycles}
The authors of \cite{GHS13:matter} associate to the fibrations $\widetilde{q}$  ordered collections of vanishing cycles.
In what follows, 
we show how to obtain the collections of \cite{GHS13:matter}
for the $E_\ell$ cases 
from our collections \eqref{eq:vc-dP1}, \eqref{eq:vc-dP2}, and \eqref{eq:vc-dP3}. 

From \eqref{eq:vc-dP1} we remove the cycle in position $0$, since the  critical value $\lambda_0$ is far from the origin. 
Then, we change basis to $A,B$ such that $A=a, A+B=b$. This produces (up to re-orienting $b$) 
the sequence of vanishing cycles $Z_{E_8}$ of \cite[Section 3.4]{GHS13:matter}:
\begin{equation}
Z_{E_8}\coloneqq \left( A, -(A+B),  A, -(A+B),  A, -(A+B),  A, -(A+B),  A, -(A+B) \right)
\end{equation}

From \eqref{eq:vc-dP2}, we again ignore the  vanishing cycle in position $0$ and apply the change of basis above. The result is the sequence 
\[ \left( A,A+B, A, A, -B, A+B, A+B, A, A+B \right) \]
Apply a left mutation $\sfL_2$ of $A$ across $A+B$ and a right mutation $\sfR_7$ of $A+B$ across $A$ (where the classes are indexed from $1$ to $9$). The result matches the sequence $Z_{E_7}$ of \cite[Section 3.4]{GHS13:matter} up to orientations: 
\begin{equation}
%Z_{E_7}\coloneqq \left( A, -B, A+B, A, -B, A+B, A, B, A+B \right)
Z_{E_7}\coloneqq \left( A, B, -(A+B), A, B, -(A+B), A, B, -(A+B) \right)
\end{equation}

Finally, we obtain the sequence $Z_{E_6}$ of \cite[Section 3.4]{GHS13:matter} by dropping the cycle of \eqref{eq:vc-dP3} in position $0$ and changing to a basis $b=A$, $a=-(A+B)$: 
\begin{equation}
Z_{E_6}\coloneqq \left( A, -(A+B),  A, -(A+B),  A, -(A+B),  A, -(A+B)   \right)
\end{equation}

\subsubsection{Bilinear pairings and matrix factorizations}

The paper \cite{GHS13:matter} associates to the fibration $\widetilde{q}$ a lattice, called the \textit{string
junction} lattice. As an abelian group, this lattice is the free abelian group $\bfJ$ generated by the Lefschetz thimbles of $\widetilde{q}$, it has rank $N$. It is isomorphic to $H_2(\widetilde{Z}, \Sigma;  \Z)$, where $\Sigma$ is a regular fiber of $\widetilde{q}$ \cite[Theorem 2.5]{GHS16_junctions}. The group $\bfJ$ carries a \textit{junction pairing} $\pairing_\rJ$, which is the symmetrization of the Seifert pairing defined by \eqref{eq:SeifertPairing} (see \cite[Remark 4.3]{HarderThompson20}). 
Replacing $\pairing_\rJ$ with $\pairing_{\mathrm{Sft}}$, we obtain an isometry of pseudolattices
\[ (K_0(\widetilde{q}), \chi) \simeq (\bfJ, \pairing_{\mathrm{Sft}})  \]

We provide a B-side interpretation of this pseudolattice in the $E_\ell$ case. Let $X_d$ be a del Pezzo surface 
as in Example \ref{ex:DP} with $d=9-l$.
Consider the full triangulated subcategory 
\[\!^\perp\sO_{X_d} \coloneqq  \{E\in D^b(\Coh(X_d)) \,\mid\, \Hom(E,\sO_{X_d}[i])=0 \mbox{ for all }i\} \subset D^b(\Coh(X_d))\] 
\begin{proposition}
\label{pro:orthogonal}
There is an isometry of pseudolattices:
\[ K_0(\!^\perp\sO_{X_d}) \simeq K_0(\widetilde{q}_\ell)
\]
\end{proposition}
\begin{proof}
Removing the element in position $0$ from the basis \eqref{eq:ExcBasisDP}, we obtain an exceptional basis for $K_0(\!^\perp\sO_{X_d})$. The corresponding Gram matrix is obtained from \eqref{eq:GramDP} by erasing the top row and left-most column. 
It is shown in Section \ref{ssec:GHSCycles} that removing the element in position $0$ from the bases \eqref{eq:vc-dP3},\eqref{eq:vc-dP2}, and \eqref{eq:vc-dP1} we obtain exceptional bases for $K_0(\widetilde{q}_\ell)$.

Observe, finally, that the mutations in Theorem \ref{thm:MutationSequence} never involve the $0$-th position. Thus, the same mutation sequences send the bases for $K_0(\widetilde{q}_\ell)$ to \eqref{eq:BasisAfterMutations} with the $0$-th element removed, which produces the isometry.
\end{proof}
%\cite[Theorem 3.11]{Orlov09_DerSing}
Recall that $X_d$ is a weighted projective hypersurface (Section \ref{sec:Xd-HV}), let $h^{(d)}$ %$h_d$ 
be its defining polynomial.
%generic weighted homogeneous polynomial of degree $d$ in the corresponding weighted projective space. 
Then by \cite[Remark 3.12]{Orlov09_DerSing}, there is an equivalence of categories between $\!^\perp\sO_{X_d}$ and the category of graded matrix factorizations $\mathrm{DGr}(h^{(d)})$. Then, directly from Proposition \ref{pro:orthogonal} we obtain the following corollary.
\begin{corollary}\label{cor:MF}
$K_0(\widetilde{q}_\ell)\simeq K_0(\mathrm{DGr}(h^{(d)}))$, and the isomorphism respects the bilinear pairings.  
\end{corollary}

\subsubsection{Representation theory of $E_\ell$ through the B-side}

The representation theory of the ADE Lie algebras 
%associated with $E_\ell$ 
goes through the theory of abstract root systems and weight lattices. We direct readers interested to the general theory to \cite[Chapters I-III]{Humphreys}. Here, we only recall some aspects of it, in the simplified version afforded by working with ADE root systems. 

Suppose that $R$ is a finite root system of type ADE. 
%To $R$ is associated a (symmetric) Cartan matrix $A$. 
The root lattice $Q=Q(R)$ of $R$ is the $\Z$-span of any root basis. The choice of a root basis induces an isomorphism $Q\simeq \Z^\ell$ as abelian groups. %$Q\simeq \Z^r$. %, and the (symmetric) Gram matrix $A$ of the pairing written in the standard basis. 
The \textit{weight lattice} is the dual lattice $Q^\vee \coloneqq \Hom_\Z(Q,\Z)$, and a basis of $Q^\vee$ is called a basis of \textit{fundamental weights} if it is dual to a root basis of $Q$.\footnote{The terminology here is in accordance with Lie theory: for ADE root systems, the set of roots and the set of coroots coincide. Thus, the fundamental weights defined above are also dual to coroots (see \cite[III.13]{Humphreys}).} 

\smallskip

We briefly recall the approach to Lie theory through string junctions. This is first laid out in \cite{DWZ}, who construct their junction lattices with algebraic methods, and further developed in \cite{GHS13:matter}, which obtains $\bfJ$ geometrically as described above.\footnote{Without explicitly showing that the two constructions are isomorphic, \cite{GHS13:matter} nevertheless illustrates that the representation theoretic consequences are consistent.}
An important ingredient for the construction is the \textit{asymptotic charge}. 
For $w\colon M\to C$ a genus-$1$ Lefschetz fibration with general fiber $\Sigma$, this is the homomorphism 
\[c_w\colon H_2(M, \Sigma ; \Z) \to H_1(\Sigma,\Z)\]
sending the class of a thimble $D$ to the class of its boundary $\partial D$.

In the ADE setting, the kernel of the asymptotic charge $c_{\widetilde{q}}$ contains a copy of the root lattice $Q$ as a sublattice \cite[Sections 3.3-4]{DWZ}.
Combined with this, the canonical primitive embedding $Q\hookrightarrow Q^\vee$ induces an embedding 
\[
 Q^\vee \hookrightarrow \bfJ_\Q \coloneqq \bfJ\otimes \Q\]
 whose image spans the kernel of $(c_{\widetilde{q}})_\Q$ (the rational linear extention of $c_{\widetilde{q}}$ to $\bfJ_\Q$).
Since $N>\ell$, fundamental weights $\omega_1,\ldots,\omega_\ell$ do not span the whole $\bfJ_\Q$. 
Then, \cite[\S~4]{DWZ} constructs additional \textit{extended weights} $\omega_p,\omega_q\in \bfJ_\Q$ (or just $\omega_p$ in the $A_n$ case), whose charges span $\im((c_{\widetilde{q}})_\Q)\simeq H_1(\Sigma,\Q)$, in such a way that in the basis $(\{\omega_i\},\omega_p,\omega_q)$ the bilinear pairing splits as
\begin{equation}\label{eq:splitPairingJunction}
    \pairing_\bfJ = \pairing_{Q^\vee} \oplus \, \pairing_c
\end{equation}
where $\pair{u,v}_c$ is a bilinear form depending only on $c_{\widetilde{q}}(u),c_{\widetilde{q}}(v)$. 
In other words, they construct an orthogonal splitting 
\[ \begin{split} \bfJ_\Q & \simeq \ker((c_{\widetilde{q}})_\Q) \oplus \im((c_{\widetilde{q}})_\Q) \\ & \simeq Q^\vee_\Q \oplus H_1(\Sigma, \Q)
\end{split} \]

We focus on the $E_\ell$ case and provide a convenient algebro-geometric interpretation on the B-side of the above discussion above. This is done by considering 
the slightly larger pseudolattices $H_2(\Yt_d,\Sigma; \Z)  = K_0(\wt_d) \simeq K_0(X_d)$, where $d=9-\ell$ , defined in Section \ref{sec:surfacelikePseudolatt-prelim}.

Before we explain this, we recall the notation introduced in Example \ref{ex:DP}.
The $\sfE_\ell$ root lattice is isomorphic to the sublattice of $\sfI^{1,\ell}$ orthogonal to $\bfk_\ell$, and $\sfI^{1,\ell}\simeq \NS(X_d)$, with $\bfk_\ell$ corresponding to the class of a canonical divisor. 
For the rest of the Section, we fix $\ell$, and omit $\ell$ and $d$ from the notation whenever they are clear from the context.

Let $i\colon C \to X$ denote the inclusion of a smooth anticanonical divisor, and fix the symplectic basis $([\sO_C], [\sO_x])$ (here $x\in C$) for the numerical Grothendieck group $K_0^{\mathrm{num}}(C)$.
Consider the restriction map $i^*\colon K_0(X) \rightarrow K_0^{\mathrm{num}}(C)$.

\begin{proposition}
\label{prop:KerIstar}
Let $\bfp$ denote the point like object $[\sO_x]\in K_0(X)$. Then 
\[\Ker(i^*) \simeq \sfE_\ell \oplus \langle \bfp\rangle\]
\end{proposition}

\begin{proof}
Identify $ K_0(X)\simeq \langle [\sO_X] \rangle \oplus \NS(X) \oplus \langle \bfp \rangle$. Then, for $v=r[\sO_X] + \bfe + d\bfp \in K_0(X)$ (here $\bfe \in \NS(X)$), the restriction $i^*$ writes
\[ i^*(v) = r[\sO_C] - (\bfe \cdot \bfk)[\sO_x]\]
since $C$ has class $-\bfk$. Then, the kernel of $i^*$ consists of rank zero classes in $K_0(X)$ orthogonal to $\bfk$. 
Thus, $\ker(i^*) \simeq \sfE_\ell \oplus \bfp$.
\end{proof}

By Theorem \ref{thm:MutationSequence} and \cite[Theorem 5.1]{HarderThompson20}, the restriction map $i^\ast$ is the B-side counterpart of the asymptotic charge $c_{\wt}\colon H^2(\widetilde{Y},\Sigma;\Z) \to H^1(\Sigma,\Z)$.  Abusing notation, denote again by $\bfp$ the image of $\bfp$ under the chosen isomorphism $K_0(X) \simeq K_0(\widetilde{w})$. 
It follows that:

\begin{corollary}
\label{cor:kernel-dec}
    The kernel $\Ker(c_{\wt})$ decomposes as an orthogonal sum
    \[\sfE_{\ell} \oplus \langle \bfp \rangle\]
\end{corollary}

The Corollary %\ref{prop:KerIstar} 
is a stronger version, for the pseudolattice $K_0(\widetilde{w})$, of the statement of \cite{DWZ} that $\ker(c_{\widetilde{q}})$ contains a copy of $Q=\sfE_\ell$.

\smallskip

To introduce an analog of the extended weights and of the splitting \eqref{eq:splitPairingJunction}, we work on the B-side and give a basis of $K_0(X)_\Q$ which splits off the image of $i^*$, and consider the associated Gram matrix of the pairing. 

Let $\beta_1,\ldots,\beta_\ell$ be a root basis of  $\sfE_\ell$. We may regard these as embedded in $\bfk^\perp \subset \sfI^{1,\ell}$. Then, $(\bfk,\beta_1,\ldots,\beta_\ell)$ is a basis of $\sfI^{1,\ell}_\Q\simeq \NS(X)_\Q$, and
\begin{equation}\label{eq:basisKuz}
([\sO_X], \bfk, \beta_1,\ldots, \beta_\ell,\bfp)   
\end{equation}
is a (rational) basis of $K_0(X)_\Q$.
Observe that $\mathrm{Span}_\Q([\sO_X],\bfk)$ is isomorphic to $\im(i^*_\Q)$. 

We may use Rieman--Roch to compute the Gram matrix of the Euler pairing in the basis $\eqref{eq:basisKuz}$. Indeed, since $\bfk$ is the class of a canonical divisor, we have $\td_X = (1,-\frac{\bfk}{2}, 1)$. Then 
\begin{align*}
    \chi(\sO_X,\bfk)=\int (0,\bfk,0)\td_X = -\frac{\bfk^2}{2} = -\frac{d}{2} && \qquad
    \chi(\sO_X,\beta_i) = - \frac{\beta_i \cdot \bfk}{2}=0 \\
    \chi(\bfk,\sO_X)= - \chi(\sO_X,\bfk) = \frac{d}{2} && \qquad
\chi(\beta_i,\sO_X)=-\chi(\sO_X,\beta_i)=0
\end{align*}
for $i=1,\ldots,\ell$. Let $A_\ell$ be the Cartan matrix of $\sfE_\ell$, i.e. the Gram matrix of the root pairing among the $\beta_i$. The Gram matrix of the bilinear form is:

\begin{equation}\label{eq:GramKuz}
{%\scriptsize
\left(\begin{array}{c|c|c}
\begin{array}{cc}
   1  & -\frac{d}{2} \\
  \frac{d}{2}   & d
\end{array}
&\begin{array}{ccc}
     0 & \ldots & 0 \\
     0 & \ldots & 0 
\end{array}  & \begin{array}{c}
     1  \\
     0 
\end{array}  \\
\hline
\begin{array}{cc}
      0 & 0 \\ \vdots & \vdots \\ 0 & 0
\end{array}&
    
        \text{\large $A_{\ell}$}
     &
    \begin{array}{c}
      0 \\ \vdots \\ 0  
\end{array}\\
\hline
    \begin{array}{cc}
        1 & 0 
    \end{array} & 
    \begin{array}{ccc}
      0 & \ldots & 0  \\
\end{array}
    &
     0
    \end{array}\right)
    }
\end{equation}
\smallskip

Using Theorem \ref{thm:MutationSequence} and \cite[Theorem 5.1]{HarderThompson20}, we may rephrase this description in terms of $K_0(\widetilde{w})$ as follows.

\begin{proposition}
\label{pro:H-splitting} 
The pseudolattice $K_0(\widetilde{w})$  admits a splitting 
\begin{equation}
    \label{eq:splitting}
    K_0(\widetilde{w})_\Q \simeq \Ker((c_{\wt})_\Q) \oplus \im((c_{\wt})_\Q)
\end{equation}
which, modulo $\bfp$, is orthogonal with respect to the Seifert pairing.    
\end{proposition}

Incidentally, we remark that $\sfI^{1,\ell}\simeq \NS(K_0(\widetilde{w}_d))$ contains fundamental weights for $\sfE_\ell$, and hence a copy of the weight lattice \cite[\S~8.2.4]{Dol_CAG}. Thus the expectation, formulated in \cite[\S~3.5]{GHS13:matter}, that string junctions encode the full structure of ADE algebras, and hence reproduce arbitrary representations and their weights, is confirmed in the pseudolattices $K_0(\widetilde{w}_d)$.

\section{Categorical lifts}\label{sec:lifts}

The aim of this section is to lift the isometry of Theorem \ref{thm:MutationSequence} to an equivalence of Fukaya--Seidel categories. To do so, recall the LG models $(\Yt,\wt)$ (introduced in \eqref{eq:TildeYw}) and $(M,W)$ from \cite{AKO06} (see Section \ref{sec:geometry-cats}). We drop the indices $d$ and $\ell$ from the notation. For ease of exposition, in this section we denote by $(N,p)$ the LG model $(\Yt,\wt)$, and by $(N',p')$ the LG model $(M,W)$. 

Since it is built from Weierstrass forms, $(N,p)$ comes with a distinguished section $D$. Let $\Omega$ be a symplectic form on $N$ such that $[\Omega]$ is Poincar\'e dual to $[D]$ (the brackets $[-]$ denote (co)homology classes). 
Similarly, consider $D'$ and $\Omega'$ for $(N',p')$. 
The manifolds $G\coloneqq N \setminus D$ and $G'\coloneqq N^\prime\setminus D'$, equipped with the restrictions of the maps and the chosen symplectic forms, are exact symplectic Lefschetz fibrations.

We may perturb $\Omega$ without changing its cohomology class so that,
in a tubular neighborhood  $t(D)$ of $D$, a symplectic parallel transport is well-defined. 
Arguing as in \cite[Section (3A)]{Seidel-more}, this produces a trivialization of  $t(D)$. 
Let $E \coloneqq N\setminus t(D)$. The restriction $\pi\coloneqq p_{|E}$ gives rise to an exact symplectic Lefschetz fibration which is trivial near $\partial_h E$ (this denotes the horizontal boundary of $E$, as in \cite[Section (15a)]{Seidel}). 
We write $\omega$ for the restriction of $\Omega$ to $E$.
Likewise, from $(N',p',\Omega')$ we obtain an exact symplectic Lefschetz fibration $(E',\pi',\omega')$, trivial near $\partial_h E'$.

We associate full exceptional collections of vanishing cycles to $(E,\pi,\omega)$ and $(E',\pi',\omega')$ with the construction in Section \ref{sec:construct-vc}.
Let $K$ and $K'$ denote reference fibers of $\pi$ and $\pi'$. 

Proceeding as in \cite[Section 3.3]{AKO06}, we obtain cycles $(L_0',\ldots,L_{\ell+2}')$ on $K'$ with cohomology classes as in \cite[Lemma 3.1]{AKO06}.
Choosing reference fiber and vanishing paths for $(E,\pi,\omega)$ as in Section \ref{sec:construct-vc} yields a collection of vanishing cycles whose cohomology classes are in Proposition \ref{pro:coho-classes}. 
Applying to this collection the sequences of mutations of Theorem \ref{thm:MutationSequence} (and Remark \ref{rem:change-basis}), we obtain a choice of vanishing paths for $(E,\pi,\omega)$, and hence an exceptional collection of vanishing cycles $(L_0,\ldots,L_{\ell+2})$ on $K$, such that for all $i=0, \dots, l+2$
\begin{equation}
\label{eq:homology} \varphi[L_i]=[L_i']
\end{equation}
under a suitable isomorphism $\varphi \colon H_1(K,\Z)\to H_1(K',\Z)$.

\begin{thm}
    \label{thm:equivOfCategories}
There is an equivalence of $A_\infty$-categories
\[ \Phi \colon D^b\lagvc(E,\pi,\omega) \xrightarrow{\sim} D^b\lagvc(E',\pi',\omega') \]
sending the full exceptional collection $(L_0\ldots,L_{\ell+2})$ to the full exceptional collection $(L'_0,\ldots,L'_{\ell+2})$.
\end{thm}

\begin{proof}
Observe that $K$ and $K'$ are exact symplectic manifolds with the restrictions  $\kappa$ and $\kappa'$ of $\omega$ and $\omega'$. 
We write $\theta$ (resp. $\theta'$) for a primitive of $\kappa$ (resp. $\kappa'$). By construction, the vanishing cycles $L_i$ (resp. $L_i'$) are exact Lagrangians with respect to $\theta$ (resp. $\theta'$), i.e. $\theta_{L_i}=dh_i$ and $\theta_{L_i'}=dh_i'$ for smooth functions $h_i,h_i'$.

The map $\varphi$ lifts to a diffeomeorphism $\phi \colon K \to K'$ so that $\varphi = \phi_*$, and we have $[\kappa]=\phi^*[\kappa']$. 
With a standard application of Moser's trick, we may assume that 
\begin{equation}\label{eq:Moser} 
\theta - \phi^* \theta' = dg
\end{equation} 
is an exact form. This implies in particular that $\phi$ is a symplectomorphism. 
Moreover, it guarantees that $L_i$ and $\phi^{-1}L_i'$ are Hamiltonian isotopic, as by \eqref{eq:homology} they are homologous and by \eqref{eq:Moser} they bound zero area.

Then, by the uniqueness part of \cite[Lemma 16.9]{Seidel} we obtain an isomorphism $(E,\pi,\omega) \xrightarrow{\sim} (E',\pi',\omega')$ (defined as in \cite[Section (15b)]{Seidel}) and hence the desired equivalence.
\end{proof}

We can relate the Fukaya--Seidel categories of Theorem \ref{thm:equivOfCategories} with those associated to the manifolds $G$ and $G'$. Let $q\coloneqq p_{|G}$ and $\varpi \coloneqq \Omega_{|G}$, and similarly for $q'$ and $\varpi'$.
Observe that $(L_0,\ldots,L_{\ell+2})$ is a full exceptional collection for  $D^b\lagvc(G,q,\varpi)$, and likewise $(L'_0,\ldots,L'_{\ell+2})$ for $D^b\lagvc(G',q',\varpi')$. 

\begin{corollary}
\label{cor:equivOfCategories}
There is an equivalence of $A_\infty$-categories
\[ D^b\lagvc(G,q,\varpi) \xrightarrow{\sim} D^b\lagvc(G',q',\varpi') \]
sending the full exceptional collection $(L_0\ldots,L_{\ell+2})$ to the full exceptional collection $(L'_0,\ldots,L'_{\ell+2})$.
\end{corollary}

\begin{proof}
An equivalence $\alpha\colon D^b\lagvc(G,q,\varpi) \simeq D^b\lagvc(E,\pi,\omega)$ follows from the fact that Lagrangians in $G$ are Hamiltonian isotopic to Lagrangians in $E$, and holomorphic discs in $G$ with boundary in $E$ must be entirely contained in $E$ by the maximum principle.
The desired equivalence is then the composition $(\alpha')^{-1} \circ \Phi \circ \alpha$, where $\Phi$ is as in Theorem \ref{thm:equivOfCategories} and $\alpha' \colon D^b\lagvc(G',q',\varpi') \simeq D^b\lagvc(E',\pi',\omega')$ is analogous to $\alpha$.
\end{proof}

\begin{remark}
    \label{rem:compare-HK}
For a log Calabi--Yau surface $(Y,D)$ with maximal boundary and distinguished complex structure, \cite[Theorem 1.1]{HackingKeating}  constructs a mirror symplectic Lefschetz fibration over $\CC$ with fibre a $k$-punctured elliptic curve, where $k$ is the number of components of $D$. 

In the language of \cite[Section 6.4]{HackingKeating}, the exact Lefschetz fibration $(G',q',\varpi')$ is expected to be mirror to 
the deformation with distinguished complex structure of a log CY surface given by %$(X,D)$, where $X$ is 
a degree $d$ del Pezzo surface and an elliptic curve with a single node.
\end{remark} 

\begin{remark}\label{rem:BulkDef}
Removing the section $D$ and $D'$ from $N$ and $N'$ puts us the exact setting and avoids symplectic area computations. 

The relation between the Fukaya--Seidel categories associated to $(G,q)$ and $(N,p)$ is quite subtle. 
By \cite[Theorem 1.2]{SeidelAinfII12}, the category $D^b\lagvc(G,q,\varpi)$ is equivalent to a specific bulk deformation of $D^b\lagvc(N,p,\Omega)$. This deformation is trivial for the mirrors of $\pr 2$ and toric del Pezzo surfaces \cite[Section (5.b)]{SeidelAinfII12}. 
\end{remark}

%\bibliographystyle{alpha}
%\bibliography{bibliography}

\begin{thebibliography}{dTdVVdB16}

\bibitem[ACC{\etalchar{+}}16]{Cor15}
Mohammad Akhtar, Tom Coates, Alessio Corti, Liana Heuberger, Alexander
  Kasprzyk, Alessandro Oneto, Andrea Petracci, Thomas Prince, and Ketil
  Tveiten.
\newblock Mirror symmetry and the classification of orbifold del {P}ezzo
  surfaces.
\newblock {\em Proc. Amer. Math. Soc.}, 144(2):513--527, 2016.

\bibitem[AGZV12]{AGZVII}
V.~I. Arnold, S.~M. Gusein-Zade, and A.~N. Varchenko.
\newblock {\em Singularities of differentiable maps. {V}olume 2}.
\newblock Modern Birkh\"auser Classics. Birkh\"auser/Springer, New York, 2012.
\newblock Monodromy and asymptotics of integrals, Translated from the Russian
  by Hugh Porteous and revised by the authors and James Montaldi, Reprint of
  the 1988 translation.

\bibitem[AKO06]{AKO06}
Denis Auroux, Ludmil Katzarkov, and Dmitri Orlov.
\newblock Mirror symmetry for del {P}ezzo surfaces: vanishing cycles and
  coherent sheaves.
\newblock {\em Invent. Math.}, 166(3):537--582, 2006.

\bibitem[Bri71]{Brieskorn}
E.~Brieskorn.
\newblock Singular elements of semi-simple algebraic groups.
\newblock In {\em Actes du {C}ongr\`es {I}nternational des {M}ath\'ematiciens
  ({N}ice, 1970), {T}ome 2}, pages 279--284. Gauthier-Villars \'Editeur, Paris,
  1971.

\bibitem[CCG{\etalchar{+}}13]{CC}
Tom Coates, Alessio Corti, Sergey Galkin, Vasily Golyshev, and Alexander
  Kasprzyk.
\newblock Mirror symmetry and {F}ano manifolds.
\newblock In {\em European {C}ongress of {M}athematics}, pages 285--300. Eur.
  Math. Soc., Z\"{u}rich, 2013.

\bibitem[CCGK16]{Quantum-P-3d}
Tom Coates, Alessio Corti, Sergey Galkin, and Alexander Kasprzyk.
\newblock Quantum periods for 3-dimensional {F}ano manifolds.
\newblock {\em Geom. Topol.}, 20(1):103--256, 2016.

\bibitem[CG07]{Coates-Givental}
Tom Coates and Alexander Givental.
\newblock Quantum {R}iemann-{R}och, {L}efschetz and {S}erre.
\newblock {\em Ann. of Math. (2)}, 165(1):15--53, 2007.

\bibitem[CG21]{ACGG}
Alessio Corti and Giulia Gugiatti.
\newblock Hyperelliptic integrals and mirrors of the {J}ohnson-{K}oll\'{a}r del
  {P}ezzo surfaces.
\newblock {\em Trans. Amer. Math. Soc.}, 374(12):8603--8637, 2021.

\bibitem[CKP19]{LInversion}
Tom Coates, Alexander Kasprzyk, and Thomas Prince.
\newblock Laurent inversion.
\newblock {\em Pure Appl. Math. Q.}, 15(4):1135--1179, 2019.

\bibitem[CKPT21]{CKPT}
Tom Coates, Alexander~M. Kasprzyk, Giuseppe Pitton, and Ketil Tveiten.
\newblock Maximally mutable {L}aurent polynomials.
\newblock {\em Proc. A.}, 477(2254):Paper No. 20210584, 21, 2021.

\bibitem[CLS11]{cox2011toric}
David~A Cox, John~B Little, and Henry~K Schenck.
\newblock {\em Toric varieties}, volume 124.
\newblock American Mathematical Soc., 2011.

\bibitem[CMSP17]{Carlson}
James Carlson, Stefan M\"uller-Stach, and Chris Peters.
\newblock {\em Period mappings and period domains}, volume 168 of {\em
  Cambridge Studies in Advanced Mathematics}.
\newblock Cambridge University Press, Cambridge, second edition, 2017.

\bibitem[Dim92]{Dimca_SingularitiesAndTopology}
Alexandru Dimca.
\newblock {\em Singularities and topology of hypersurfaces}.
\newblock Universitext. Springer-Verlag, New York, 1992.

\bibitem[DKK12]{DKK12}
Colin {Diemer}, Ludmil {Katzarkov}, and Gabriel {Kerr}.
\newblock {Symplectomorphism group relations and degenerations of
  Landau-Ginzburg models}.
\newblock {\em arXiv e-prints}, page arXiv:1204.2233, April 2012.

\bibitem[Dol12]{Dol_CAG}
Igor~V. Dolgachev.
\newblock {\em Classical algebraic geometry}.
\newblock Cambridge University Press, Cambridge, 2012.
\newblock A modern view.

\bibitem[dTdVVdB16]{VdB_dTdV}
Louis de~Thanhoffer~de Volcsey and Michel Van~den Bergh.
\newblock {On an analogue of the Markov equation for exceptional collections of
  length 4}.
\newblock {\em arXiv e-prints}, page arXiv:1607.04246, July 2016.

\bibitem[DZ99]{DWZ}
Oliver DeWolfe and Barton Zwiebach.
\newblock String junctions for arbitrary {L}ie-algebra representations.
\newblock {\em Nuclear Phys. B}, 541(3):509--565, 1999.

\bibitem[{Fri}15]{Friedman15}
Robert {Friedman}.
\newblock {On the geometry of anticanonical pairs}.
\newblock {\em arXiv e-prints}, page arXiv:1502.02560, February 2015.

\bibitem[GGLP23]{GGLP}
Antonella Grassi, Giulia Gugiatti, Wendelin Lutz, and Andrea Petracci.
\newblock Reflexive polygons and rational elliptic surfaces.
\newblock {\em Rendiconti del Circolo Matematico di Palermo Series 2}, 72:3185
  -- 3221, 2023.

\bibitem[GHS13]{GHS13:matter}
Antonella Grassi, James Halverson, and Julius~L. Shaneson.
\newblock Matter from geometry without resolution.
\newblock {\em J. High Energy Phys.}, (10):205, front matter+44, 2013.

\bibitem[GHS16]{GHS16_junctions}
Antonella Grassi, James Halverson, and Julius~L. Shaneson.
\newblock Geometry and topology of string junctions.
\newblock {\em J. Singul.}, 15:36--52, 2016.

\bibitem[Giv98]{GIV3}
Alexander Givental.
\newblock A mirror theorem for toric complete intersections.
\newblock In {\em Topological field theory, primitive forms and related topics
  ({K}yoto, 1996)}, volume 160 of {\em Progr. Math.}, pages 141--175.
  Birkh\"{a}user Boston, Boston, MA, 1998.

\bibitem[GMRV]{GMV}
Giulia Gugiatti, Martin Mereb, and Fernando Rodriguez~Villegas.
\newblock Hypergeometric elliptic surfaces.
\newblock To appear.

\bibitem[GR23]{GR22}
Giulia Gugiatti and Franco Rota.
\newblock Full exceptional collections for anticanonical log del {P}ezzo
  surfaces.
\newblock {\em Int. Math. Res. Not. IMRN}, 2023(21):18803--18855, 2023.

\bibitem[GT24]{GiovenzanaThompson}
Luca {Giovenzana} and Alan {Thompson}.
\newblock {Degenerations and Fibrations of K3 Surfaces: Lattice Polarisations
  and Mirror Symmetry}.
\newblock {\em arXiv e-prints}, page arXiv:2405.12009, May 2024.

\bibitem[HIV00]{HV-Lag}
Kentaro {Hori}, Amer {Iqbal}, and Cumrun {Vafa}.
\newblock {D-Branes And Mirror Symmetry}.
\newblock {\em arXiv e-prints}, pages hep--th/0005247, May 2000.

\bibitem[HK22]{HackingKeating}
Paul Hacking and Ailsa Keating.
\newblock Homological mirror symmetry for log {C}alabi-{Y}au surfaces.
\newblock {\em Geom. Topol.}, 26(8):3747--3833, 2022.
\newblock With an appendix by Wendelin Lutz.

\bibitem[HT20]{HarderThompson20}
Andrew Harder and Alan Thompson.
\newblock Pseudolattices, del {P}ezzo surfaces, and {L}efschetz fibrations.
\newblock {\em Trans. Amer. Math. Soc.}, 373(3):2071--2104, 2020.

\bibitem[Hum72]{Humphreys}
James~E. Humphreys.
\newblock {\em Introduction to {L}ie algebras and representation theory},
  volume Vol. 9 of {\em Graduate Texts in Mathematics}.
\newblock Springer-Verlag, New York-Berlin, 1972.

\bibitem[HV00]{Hori-Vafa}
Kentaro {Hori} and Cumrun {Vafa}.
\newblock {Mirror Symmetry}.
\newblock {\em arXiv e-prints}, pages hep--th/0002222, February 2000.

\bibitem[JK01]{JK}
Jennifer~M. Johnson and J\'{a}nos Koll\'{a}r.
\newblock K\"{a}hler--{E}instein metrics on log del {P}ezzo surfaces in
  weighted projective 3-spaces.
\newblock {\em Ann. Inst. Fourier (Grenoble)}, 51(1):69--79, 2001.

\bibitem[Kat91]{Katz}
Sheldon Katz.
\newblock Small resolutions of {G}orenstein threefold singularities.
\newblock In {\em Algebraic geometry: {S}undance 1988}, volume 116 of {\em
  Contemp. Math.}, pages 61--70. Amer. Math. Soc., Providence, RI, 1991.

\bibitem[Kea15]{KeatingTori}
Ailsa Keating.
\newblock Lagrangian tori in four-dimensional {M}ilnor fibres.
\newblock {\em Geom. Funct. Anal.}, 25(6):1822--1901, 2015.

\bibitem[Kea18]{KeatingCusp}
Ailsa Keating.
\newblock Homological mirror symmetry for hypersurface cusp singularities.
\newblock {\em Selecta Math. (N.S.)}, 24(2):1411--1452, 2018.

\bibitem[KM98]{KM98}
J\'{a}nos Koll\'{a}r and Shigefumi Mori.
\newblock {\em Birational geometry of algebraic varieties}, volume 134 of {\em
  Cambridge Tracts in Mathematics}.
\newblock Cambridge University Press, Cambridge, 1998.
\newblock With the collaboration of C. H. Clemens and A. Corti, Translated from
  the 1998 Japanese original.

\bibitem[Kon95]{K1994}
Maxim Kontsevich.
\newblock Homological algebra of mirror symmetry.
\newblock In {\em Proceedings of the {I}nternational {C}ongress of
  {M}athematicians, {V}ol. 1, 2 ({Z}\"{u}rich, 1994)}, pages 120--139.
  Birkh\"{a}user, Basel, 1995.

\bibitem[Kon98]{Kon_Lectures}
Maxim Kontsevich.
\newblock Lectures at {ENS}, {P}aris, 1998.
\newblock notes taken by J. Bellaiche, J.-F. Dat, I. Marin, G. Racinet and H.
  Randriambololona. Unpublished, 1998.

\bibitem[Kuz17]{Kuz17_pseudolattices}
A.~G. Kuznetsov.
\newblock Exceptional collections in surface-like categories.
\newblock {\em Mat. Sb.}, 208(9):116--147, 2017.

\bibitem[Loo81]{Loo81}
Eduard Looijenga.
\newblock Rational surfaces with an anticanonical cycle.
\newblock {\em Ann. of Math. (2)}, 114(2):267--322, 1981.

\bibitem[Mir89]{Miranda-ES}
Rick Miranda.
\newblock {\em The basic theory of elliptic surfaces}.
\newblock Dottorato di Ricerca in Matematica. [Doctorate in Mathematical
  Research]. ETS Editrice, Pisa, 1989.

\bibitem[MP86]{miranda_persson}
Rick Miranda and Ulf Persson.
\newblock On extremal rational elliptic surfaces.
\newblock {\em Math. Z.}, 193(4):537--558, 1986.

\bibitem[OP18]{OP}
Alessandro Oneto and Andrea Petracci.
\newblock On the quantum periods of del {P}ezzo surfaces with {$\frac 13(1,1)$}
  singularities.
\newblock {\em Adv. Geom.}, 18(3):303--336, 2018.

\bibitem[Orl09]{Orlov09_DerSing}
Dmitri Orlov.
\newblock Derived categories of coherent sheaves and triangulated categories of
  singularities.
\newblock In {\em Algebra, arithmetic, and geometry: in honor of {Y}u. {I}.
  {M}anin. {V}ol. {II}}, volume 270 of {\em Progr. Math.}, pages 503--531.
  Birkh\"{a}user Boston, Boston, MA, 2009.

\bibitem[Sei01a]{Seidel-more}
Paul Seidel.
\newblock More about vanishing cycles and mutation.
\newblock In {\em Symplectic geometry and mirror symmetry ({S}eoul, 2000)},
  pages 429--465. World Sci. Publ., River Edge, NJ, 2001.

\bibitem[Sei01b]{Seidel_VanishingCycles}
Paul Seidel.
\newblock Vanishing cycles and mutation.
\newblock In {\em European {C}ongress of {M}athematics, {V}ol. {II}
  ({B}arcelona, 2000)}, volume 202 of {\em Progr. Math.}, pages 65--85.
  Birkh\"auser, Basel, 2001.

\bibitem[Sei08]{Seidel}
Paul Seidel.
\newblock {\em Fukaya categories and {P}icard-{L}efschetz theory}.
\newblock Zurich Lectures in Advanced Mathematics. European Mathematical
  Society (EMS), Z\"{u}rich, 2008.

\bibitem[Sei16]{SeidelAinfII12}
Paul Seidel.
\newblock Fukaya {$A_\infty$}-structures associated to {L}efschetz fibrations.
  {II} 1/2.
\newblock {\em Adv. Theor. Math. Phys.}, 20(4):883--944, 2016.

\bibitem[SS10]{SS}
Matthias Sch\"utt and Tetsuji Shioda.
\newblock Elliptic surfaces.
\newblock In {\em Algebraic geometry in {E}ast {A}sia---{S}eoul 2008},
  volume~60 of {\em Adv. Stud. Pure Math.}, pages 51--160. Math. Soc. Japan,
  Tokyo, 2010.

\bibitem[Tve18]{Tve18}
Ketil Tveiten.
\newblock Period integrals and mutation.
\newblock {\em Trans. Amer. Math. Soc.}, 370(12):8377--8401, 2018.

\bibitem[Via17]{Vial17}
Charles Vial.
\newblock Exceptional collections, and the {N}\'eron-{S}everi lattice for
  surfaces.
\newblock {\em Adv. Math.}, 305:895--934, 2017.

\bibitem[Wan19]{Wang}
Jun Wang.
\newblock {A mirror theorem for Gromov-Witten theory without convexity}.
\newblock {\em arXiv e-prints}, page arXiv:1910.14440, October 2019.

\bibitem[Zag18]{Zagier}
Don Zagier.
\newblock The arithmetic and topology of differential equations.
\newblock In {\em European {C}ongress of {M}athematics}, pages 717--776. Eur.
  Math. Soc., Z\"urich, 2018.

\end{thebibliography}

\newcommand{\etalchar}[1]{$^{#1}$}

\end{document}